\documentclass[10pt]{amsart}

\usepackage{amsthm,amssymb,epsfig,graphicx,latexsym,mathdots}
\usepackage{graphicx,amsfonts,subfigure,dsfont,scalerel}
\usepackage{hyperref}
\usepackage[usenames, dvipsnames]{color}

\if{
\usepackage{geometry}
 \geometry{
 a4paper,
 total={170mm,257mm},
 left=20mm,
 top=20mm,
 }}\fi
 
\numberwithin{equation}{section}

\theoremstyle{plain}
\newtheorem{theorem}{Theorem}[section]
\newtheorem{lemma}[theorem]{Lemma}
\newtheorem{proposition}[theorem]{Proposition}
\newtheorem{corollary}[theorem]{Corollary}
\newtheorem{definition}[theorem]{Definition}

\theoremstyle{remark}
{
    
    \newtheorem{remark}[theorem]{Remark}

}
   \newtheorem{hypothesis}[theorem]{Hypothesis}

\newcommand{\norm}[2]{\left\lVert#1\right\rVert_{#2}}

\newcommand{\hu}{\hat{u}}

\newcommand{\hb}{\hat{B}}

\newcommand{\Uad}{\mathcal{U}_{\rm ad}}

\newcommand{\whq}{\widehat{\calq}}
\newcommand{\wtq}{\widetilde{\calq}}

\newcommand{\umin}{\check{u}}
\newcommand{\umax}{\hat{u}}

\newcommand{\Uspace}{{L^2(0,T)}^m}

\newcommand\finsquare{\hfill{$\square$}\end{proof}\medskip }

\newcommand{\blue}[1]{\textcolor{black}{#1}}

\def\dd{{\rm d}}
\newcommand{\ddt}{\frac{\rm d}{{\rm d} t} }

\def\weight(#1,#2){c_{#1,#2}}

\def\ra{{\rm ra}}

\if {

}{{\caption}}
} \fi

\def\fh{\hat{f}}

\def\uh{\hat{u}}

\def\yh{\hat{y}}

\def\fb{\bar{f}}

\def\hb{\bar{h}}

\def\pb{\bar{p}}

\def\tb{\bar{t}}
\def\ub{\bar{u}}
\def\vb{\bar{v}}
\def\wb{\bar{w}}
\def\xb{\bar{x}}
\def\yb{\bar{y}}

\def\Bb{\bar{B}}

\def\Mb{\bar{M}}
\def\Nb{\bar{N}}

\def\cala{{\mathcal  A}}
\def\calb{{\mathcal B}}

\def\calf{{\mathcal F}}

\def\call{{\mathcal L}}

\def\calp{{\mathcal P}}
\def\calq{{\mathcal Q}}

\def\calt{{\mathcal T}}
\def\calu{{\mathcal U}}

\def\cR{{\mathcal R}}

\def\D{\mathcal{D}}

\def\eps{\varepsilon}
\def\Om{{\Omega}}

\def\om{{\omega}}

\def\mub{\bar{\mu}}

\def\zetab{{\bar\zeta}}

\def\1B{{\bf  1}}

\def\argmin{\mathop{\rm argmin}}

\def\dist{\mathop{\rm dist}}

\def\dom{\mathop{{\rm dom}}}

\def\intt{\mathop{\rm int}}

\newcommand\Lip{\mathop{\rm Lip}}

\def\range{\mathop{\rm Im}}

\def\supp{\mathop{\rm supp}}

\def\Ker{\mathop{\rm Ker}}
\def\Min{\mathop{\rm Min}}

\def\half{\mbox{$\frac{1}{2}$}}

\def\1B{{\bf  1}}

\newcommand{\NN}{\mathbb{N}}
\newcommand{\RR}{\mathbb{R}}

\def\cN{\mathbb{N}}

\def\cR{\mathbb{R}}

\newcommand\be{\begin{equation}}
\newcommand\ee{\end{equation}}
\newcommand\ba{\begin{array}}
\newcommand\ea{\end{array}}
\newcommand{\bea}{\begin{eqnarray}}
\newcommand{\eea}{\end{eqnarray}}
\newcommand{\bean}{\begin{eqnarray*}}
\newcommand{\eean}{\end{eqnarray*}}

\def\rar{\rightarrow}

\def\ds{\displaystyle}
\def\disp{\displaystyle}

\def\la{\langle}
\def\ra{\rangle}

\title[Optimal control of a semilinear heat equation]
{State-constrained control-affine\\ parabolic problems II: \\Second order sufficient optimality conditions}

\author{M. Soledad Aronna}
\address{Escola de Matem\'atica Aplicada, FGV EMAp, Rio de Janeiro 22250-900, Brazil}
\email{soledad.aronna@fgv.br}

\author{Fr\'ed\'eric Bonnans}
\address{INRIA-Saclay and Centre de 
Math\'ematiques Appliqu\'ees, Ecole Polytechnique, 91128 Palaiseau, France}
\email{Frederic.Bonnans@inria.fr}  

\author{Axel Kr\"oner}
\address{Institut f\"ur Mathematik, Humboldt Universit\"at zu Berlin, 10099 Berlin, Germany;
Inria and CMAP, Ecole Polytechnique, CNRS, Universit\'e Paris Saclay, 91128 Palaiseau, France}
\email{axel.kroener@math.hu-berlin.de}  

\thanks{The first author was supported by FAPERJ (Brazil) through the Jovem Cientista do Nosso Estado Program; 
by CNPq (Brazil) and by the Alexander von Humboldt Foundation (Germany). The
second author thanks the `Laboratoire de Finance pour les Marchés de l’Energie’ for its support. The
second and third authors were supported by a public grant as part of the Investissement d’avenir
project, reference ANR-11-LABX-0056-LMH, LabEx LMH, in a joint call with Gaspard Monge Program for optimization, operations research and their interactions with data sciences.
\\
\indent This article is the continuation of the work \cite{ABK-PartI}, by the same authors, in which first and second order necessary conditions were established.}

\begin{document}

\maketitle


\begin{abstract}
In this paper we consider an optimal control problem governed by a
  semilinear heat equation with bilinear control-state terms and
  subject to control and state constraints. The state constraints are
  of integral type, the integral being with respect to the space
  variable. The control is multidimensional. The cost functional is of
  a tracking type and contains a linear term in the control variables. We derive second order
  sufficient conditions relying on the Goh transform. 
    The appendix provides an example illustrating the applicability of our results.
\end{abstract}

\vspace{5mm}



\section{Introduction}
This is the second part of two papers on necessary and sufficient optimality conditions for an optimal control problem governed by a semilinear heat equation containing bilinear terms coupling the control variables and the state, and subject to constraints on the control and state. 
While in the first part \cite{ABK-PartI}, first and second order necessary optimality conditions are shown, in this second part we derive second order sufficient optimality conditions.
The control may have several components and enters the dynamics in a bilinear term and in an affine way in the cost. This does not allow to apply classical techniques of calculus of variations to derive second order sufficient optimality conditions. Therefore, we extend techniques that were recently established in the following articles,
and  that involve the Goh transform~\cite{MR0205719} in an essential way. Aronna, Bonnans, Dmitruk and Lotito~\cite{ABDL12} obtained second order necessary and sufficient conditions for bang-singular solutions of control-affine finite dimensional systems with control bounds, results that were extended in Aronna, Bonnans and Goh \cite{MR3555384}  when adding a state constraint of inequality type.
An extension of the analysis in \cite{ABDL12} to the infinite dimensional setting was done by Bonnans \cite{BonnansA13}, for a problem concerning a semilinear heat equation subject to control bounds and without state constraints. For a quite general class of linear differential equations in Banach spaces with bilinear control-state couplings and subject to control bounds,  Aronna, Bonnans and Kr\"oner~\cite{MR3767765PlusErratum} provided second order conditions, that extended later to the complex Banach space setting \cite{ABK19}.

There exists a series of publications on second order conditions for problems governed by control-affine ordinary differential equations, we refer to references in \cite{ABK-PartI}.
\if{, starting with the early papers \cite{MR0205719} by Goh  and \cite{Kel64} by Kelley, later \cite{Dmi77} by Dmitruk, and recently  \cite{ABDL12}, already mentioned above.
In this context, the case dealing with both control and state constraints was treated in e.g. Maurer \cite{MR0464007}, McDanell and Powers \cite{McDaPow71}, Maurer, Kim and Vossen \cite{MaurerKimVossen2005}, Sch\"attler \cite{MR2253361}, and Aronna {\em et al.} \cite{MR3555384} as cited in the previous paragraph. Fore a more detailed description of the contributions in this framework, we refer to \cite{MR3555384}.
}\fi

In the elliptic framework, regarding the case we investigate here, this is, when no quadratic control term is present in the cost (or what some authors call {\em vanishing Tikhonov term}), 
 Casas in \cite{MR2974742} proved second order sufficient conditions for bang-bang optimal controls of a semilinear equation, and for one containing a bilinear coupling of control and state in the recent joint work with D. and G. Wachsmuth \cite{MR3878305}.

  Parabolic optimal control problems with state constraints are discussed in
R\"osch and Tr\"oltzsch~\cite{MR1982739},
who gave second order sufficient conditions for a linear equation with
mixed control-state constraints. In the presence of pure-state
constraints, Raymond and Tr\"oltzsch \cite{MR1739375}, and Krumbiegel
and Rehberg~\cite{MR3032877}
obtained second order sufficient conditions for a semilinear equation,
Casas, de Los Reyes, and Tr\"oltzsch~\cite{CasReyTro08} and de Los
Reyes, Merino, Rehberg and Tr\"oltzsch~\cite{RMRT08} obtained
sufficient second order conditions for  semilinear equations,
both in the elliptic and parabolic cases. The articles mentioned in this paragraph did not consider bilinear terms, and their sufficient conditions do not apply to the control-affine problems that we treat in the current work.

It is also worth mentioning the work  \cite{MR3374636} by Casas,
  Ryll and Tr\"oltzsch that provided second order conditions for a semilinear FitzHugh-Nagumo system subject to control constraints in the case of vanishing Tikhonov term.





The contribution of this paper are second order 
sufficient optimality conditions for an optimal control problem for a
semilinear parabolic equation with cubic nonlinearity, several
controls coupled with the state variable through bilinear
  terms, pointwise control constraints and state constraints
  that are integral in space.
The main challenge arises from the fact that  both the dynamics  and the cost function are affine with respect to the control, hence classical techniques are not applicable to derive second order sufficient conditions. We  rely on the Goh transform \cite{MR0205719} to derive sufficient optimality conditions for bang-singular solutions. 
In particular, the sufficient conditions are stated on a cone of directions larger than the one used for the necessary conditions.

The paper is organized as follows. 
 In Section~\ref{sec:1} the problem is stated and main assumptions are formulated. Section~\ref{sec:3} is devoted to second order necessary conditions and 
Section~\ref{suf-cond-sec} to second order sufficient conditions.

\subsection*{Notation} 
Let $\Om$ be an open subset of $\cR^n,$ $n\leq 3$, 
with $C^\infty$ boundary $\partial\Om$.
Given $p \in [1,\infty]$ and $k\in \NN$, let $W^{k,p}(\Omega)$ be
the Sobolev space of functions in $L^p(\Omega)$ with 
derivatives (here and after, derivatives w.r.t. $x\in\Om$ or w.r.t. time
are taken in the sense of distributions)
 in $L^p(\Omega)$ up to order $k.$ Let $\D(\Omega)$ be the set of $C^\infty$ functions with compact support in $\Om$.
By $W^{k,p}_0(\Omega)$ we denote the closure of 
$\D(\Omega)$ with respect to the $W^{k,p}$-topology.  Given a horizon $T>0$, we write $Q := \Om\times (0,T)$. $\norm{\cdot}{p}$ denotes the norm in $L^p(0,T),$ $L^p(\Omega)$ and $L^p(Q)$, indistinctly. When a function depends on both space and time, but the norm is computed only with respect of one of these variables, we specify both the space and domain. For example, if $y\in L^p(Q)$ and we fix $t\in (0,T),$ we write $\|y(\cdot,t)\|_{L^p(\Omega)}$.
For the $p$-norm in $\RR^m,$ for $m\in \NN,$ we use $|\cdot|_p$. 
We set $H^k(\Omega):= W^{k,2}(\Omega)$ and $H^k_0(\Omega):=
W_0^{k,2}(\Omega).$
By $W^{2,1,p}(Q)$ we mean the Sobolev space of $L^p(Q)$-functions
whose second derivative in space and first derivative in time belong
to $L^p(Q)$.
We write $H^{2,1}(Q)$ for $W^{2,1,2}(Q)$ and, setting 
$\Sigma:= \partial\Om\times (0,T),$ we define the state space as
\be
Y := \{ y\in H^{2,1}(Q); \; y=0  \text{ a.e. on $\Sigma$} \}.
\ee
If $y$ is a function over $Q$, we use $\dot y$ to denote its time
derivative in the sense of distributions.
As usual we denote the spatial gradient and the Laplacian by $\nabla$ and $\Delta$. By $\operatorname{dist}(t,I):=\inf \{\norm{t-\bar t}\;;\; \tb \in I \}$ for $I\subset \RR$, we denote the distance of $t$ to the set $I$.

\section{Statement of the problem and main assumptions}\label{sec:1}

In this section we introduce the optimal control problem and recall results on well-posedness of the state equation and existence of solutions of the optimal control problem from \cite{ABK-PartI}.

\subsection{Setting}
The {\em state equation} is given as
\be
\label{dynamics}
\left\{
\begin{split}
&\dot y(x,t) - \Delta y(x,t) +\gamma y^3(x,t) = 
f (x,t) + y(x,t) \sum_{i=0}^m u_i(t) b_i(x)\quad \text{in } Q,\\
&y=0\,\, \text{ on } \Sigma,\quad 
y(\cdot,0) = y_0\,\, \text{in } \Om,
\end{split}
\right.
\ee
with
\be
\label{HypSpaces1}
y_0\in H^1_0(\Om),\quad f\in L^2(Q),\quad b\in W^{1,\infty}(\Om)^{m+1},
\ee 
$\gamma \geq 0$, $u_0\equiv1$ is a constant, and 
$u:=(u_1,\ldots,u_m) \in L^2(0,T)^m$.
Lemma \ref{state-equ-estimA.l} below shows that for each control $u\in \Uspace,$  there is a unique associated solution
$y\in Y$  of 
\eqref{dynamics}, called the 
{\em associated state}. Let $y[u]$ denote this solution.
We consider control constraints of the form
$u\in \Uad$, where 
\be
\label{HypUad}
\Uad = \{u\in L^2(0,T)^m;\; \umin_i \leq u_i(t) \leq \umax_i,\,\, i=1,\dots,m\},
\ee
for some constants $\umin_i< \umax_i$, for $i=1,\dots,m.$
In addition, we have finitely many linear running state constraints of the 
form
\be
\label{stateconstraint}
g_j(y(\cdot,t)):=\int_\Omega c_j(x)y(x,t){\rm d}x+d_j \leq 0,
\quad \text{for } t\in [0,T],\;\;  j=1,\dots,q,
\ee
where $c_j\in H^2(\Om) \cap H^1_0(\Om)$ 
for $j=1,\dots,q$, and $d\in \cR^q$. 

We call any $(u,y[u])\in L^2(0,T)^m \times Y$ a {\em trajectory}, and if it additionally  satisfies the control and state constraints, we say it is an {\em admissible trajectory.}
The {\em cost function} is
\be
\label{cost}
\begin{split}
J(u,y) : = &\half \int_Q (y(x,t)-y_d(x))^2 {\rm d} x {\rm d} t 
\\
&+ \half \int_\Omega (y(x,T)-y_{dT} (x))^2 {\rm d} x 
+  \sum_{i=1}^m \alpha_i \int_0^T u_i(t) {\rm d} t,
\end{split}
\ee
where 
\be
\label{HypSpaces2}
y_d \in L^2(Q),\quad y_{dT}\in H^1_0(\Om),
\ee
and $\alpha \in \RR^m$.
We consider the optimal control  problem 
\be\label{P}\tag{P}
\Min_{u\in \Uad}  J(u,y[u]); \quad
\text{subject to \eqref{stateconstraint}}.  
\ee

For problem \eqref{P}, assuming it in the sequel to be feasible,
we consider two types of solutions. 

\begin{definition} We say that $(\ub,y[\ub])$ is an  {\em $L^2$-local solution} (resp.,  {\em $L^\infty$-local solution}) if there exists $\eps>0$ such that $(\ub,y[\ub])$ is a minimum among the admissible trajectories $(u,y)$ that satisfy $\|u-\ub\|_2<\eps$  (resp., $\|u-\ub\|_\infty<\eps$).
\end{definition}

The state equation is well-posed and has a solution in $Y$. Furthermore, the mapping $u\mapsto y$, $L^2(0,T)\rar Y$ is of class $C^{\infty}$. 
Since \eqref{P} has a bounded feasible set, it is easily checked that
its set of solutions of \eqref{P} is non-empty. 
For details regarding these assertions see Appendix \ref{sec:well-posed}.

\subsection{First order optimality conditions}\label{sec:first}
It is well-known that the dual of $C([0,T])$ is the set of (finite) Radon measures,
and that the action of a finite Radon measure coincides with the
Stieltjes integral associated with a bounded variation function $\mu\in BV(0,T)$. 
We may assume w.l.g. that $\mu (T)=0$, and we let  $\dd\mu$ denote the Radon measure
associated  to $\mu $. Note that if $\dd \mu$ belongs to
the set $\mathcal{M}_+(0,T)$ of nonnegative finite Radon measures   
then we may take $\mu $ nondecreasing and right-continuous. Set
\be
BV(0,T)_{0,+} := \big\{ \mu   \in BV(0,T) \text{ nondecreasing, right-continuous; } \mu (T) =0 \big\}. 
\ee

Let $(\ub,\yb)$ be an admissible trajectory of problem $(P)$. 
We say that $\mu\in BV(0,T)_{0,+}^q$ is 
{\em complementary to the state constraint} for $\yb$ if
\be
\int_0^T g_j(\yb(\cdot,t)) \dd \mu_j(t) {=\int_0^T \left( \int_\Omega c_j(x)\yb(x,t) \dd x + d_j \right) \dd \mu_j(t)} =0, \;\;j=1,\ldots,q.
\ee 
Let
$
(\beta,\mu)\in \RR_+\times BV(0,T)_{0,+}^q.
$
We say that $p \in L^\infty(0,T;H^1_0(\Om))$
is the {\em costate associated} with 
$(\ub,\yb,\beta,\mu)$, or shortly with $(\beta,\mu),$ 
if $(p,p_0)$ is solution of \eqref{costate-eq}.
As explained in Appendix \ref{costate_equation},
 $p = p^1 - \sum_{j=1}^q c_j \mu_j$
for some $p^1\in Y$. In particular $p(\cdot, 0)$ and $p(\cdot, T)$ are well-defined and it can be checked that $p_0=p(\cdot ,0)$. 

\begin{definition}
We say that the triple
$(\beta,p,\mu) \in \RR_+\times L^\infty(0,T;H^1_0(\Om))
\times BV(0,T)_{0,+}^q$ is a {\em generalized Lagrange multiplier} if it
satisfies the following {\em first-order optimality conditions:}
$\mu$ is complementary to the state constraint,
$p$ is the costate associated with $(\beta,\mu)$,
the non-triviality condition
$
(\beta,\dd \mu) \neq 0
$ 
holds and, for $i=1$ to $m$, 
defining the {\em switching function} by 
\be
\label{def-psi}
\Psi_i^p(t) := \beta\alpha_i + \int_\Om b_i(x) \yb(x,t) p(x,t) \dd x,
\quad \text{for }i=1,\ldots,m,
\ee
one has $\Psi^p \in L^\infty(0,T){^m}$ and 
\be
\label{FirstControl}
\sum_{i=1}^m \int_0^T \Psi^p_i(t) (u_i(t)-\ub_i(t))  \dd t \geq 0,\quad \text{for every } u\in \Uad.
\ee
We let $\Lambda(\ub,\yb)$ denote 
the set of generalized Lagrange multipliers associated with $(\ub,\yb)$.
If $\beta=0$ we say that the corresponding multiplier is {\em singular}. Finally,  we write
 $\Lambda_1(\ub,\yb)$ for the set of pairs $(p,\mu)$ with $(1,p,\mu)\in \Lambda(\ub,\yb)$. When the nominal solution is fixed and there is no place for confusion, we just write $\Lambda$ and $\Lambda_1.$
\end{definition}

\if{
\subsubsection{The reduced abstract problem}\label{sec:red-prob}

Set $F(u):= J(u,y[u]),$ and 
$G: \Uspace \rar C([0,T])^q$, $G(u):= g(y[u])$.
The {\em reduced problem} is 
\be
\label{RP}
\tag{RP}
\Min_{u\in \Uad} F(u); \quad G(u) \in K,
\ee
where $K:=C([0,T])_-^q$ is the closed convex cone of continuous
functions over $[0,T],$ with values in $\cR_-^q.$
Its interior is the set of functions in 
$C([0,T])^q$ with negative values.
We say that  the reduced problem \eqref{RP} is {\em qualified} at $\ub$ if:
\be
\label{qualif-RP}
\left\{ \ba{lll}
\text{there exists $u\in \Uad$ such that $v:=u-\ub$ satisfies}
\\
G(\ub)+DG(\ub) v \in \intt(K).
\ea\right.
\ee

Given a  Banach space 
$X,$ a closed convex subset $S\subseteq X$ and a point $\bar s \in S,$ 
the 
{\em normal cone} to $S$ at $\bar s$
is defined as
\be
N_S(\bar s)  := \{ x^*\in X^* ;\;\la x^*, s-\bar s\ra \leq 0,\,\,\, \text{for all } s\in S\}.
\ee
 We get the following first order conditions for our problem $(P)$:

\begin{lemma}\label{multipliers} 
{\rm (i)}
If $(\ub,y[\ub])$ is an $L^2$-local solution of $(P),$ then the associated set $\Lambda$ of multipliers is nonempty.
\\ {\rm (ii)}
If in addition the qualification condition \eqref{qualif-RP} holds at $\ub$, then there is no
singular multiplier, and
$\Lambda_1$
is nonempty and bounded in $L^{\infty}(0,T;H^1_0(\Om))\times BV(0,T)_{0,+}^{q}$. 
\end{lemma}

\begin{proof}
(i)
Let us consider the generalized Lagrangian associated with the reduced problem \eqref{RP}:
\be
\label{LangrianG}
L[\beta,\dd\mu](u):= \beta F(u) + \sum_{j=1}^q \int_0^T G_j(u)(t)\dd\mu_j(t).
\ee
Let $\ub$ be a local solution of \eqref{RP}.
By, e.g., \cite[Proposition 3.18]{MR1756264},
since $K$ has nonempty interior,
there exists a generalized Lagrange
multiplier associated with problem \eqref{RP}, that is, 
$(\beta,\dd\mu) \in\cR_+\times  N_K(G(\ub))$ such that 
\be
(\beta,\dd\mu)\neq 0\quad  \text{and}\quad  -D_uL[\beta,\dd\mu](\ub) \in
N_{\Uad}(\ub).
\ee
Due to the costate equation \eqref{costat-eq}, the latter condition
is equivalent to 
\eqref{FirstControl}.
\\ (ii)
That $\Lambda_1$
is nonempty and weakly-* compact 
follows from \cite[Proposition 3.16]{MR1756264}.
Now let $(p_{\ell}, \dd\mu_{\ell})$
be a bounded sequence in $\Lambda_1$.
We can assume, up to the extraction
of a subsequence, that
$\dd\mu_{\ell}$
weakly-* converges to some measure
$\dd\mub$. Also $p_{\ell}$ converges \mbox{weakly-*}, since the mapping
$(\beta,\dd\mu)\mapsto p$
($p$ being the solution of the costate equation with
data $(\beta,\dd\mu)$)
is linear continuous
and, therefore,   weakly-* continuous from
$\RR\times \mathcal{M}([0,T])$ to $L^\infty(0,T; H^1_0(\Om))$. 
Since $\Lambda_1$ is weakly-*compact, 
the conclusion follows.
\end{proof}

Observe that the qualification condition for \eqref{RP} given in \eqref{qualif-RP} holds if and only if the following qualification condition for the original problem \eqref{P} is satisfied:
\be
\label{qualif-P}
\left\{ \ba{lll}
\text{there exists $\eps>0$ and $u\in \Uad$ such that $v:=u-\ub$ satisfies}
\\
g_j(\yb(\cdot,t)) + g_j'(\yb(\cdot,t))z[v](\cdot, t) <-\eps,
\text{ for all $t\in [0,T]$, and $j=1,\dots,q$.}
\ea\right.
\ee
 In view of Lemma \ref{multipliers}, 
if \eqref{qualif-P} is satisfied, then 
$\Lambda_1$ is nonempty and weakly-* compact.

In the sequel of this section, we consider
$(\ub,\yb,\beta,p,\dd \mu),$ with 
$\yb$ the state associated with the admissible
control $\ub$ and
$(\beta,p,\dd \mu)\in \Lambda.$ 
}\fi

We recall from \cite[Lem. 3.5(i)]{ABK-PartI} the following  statement on first order conditions.
\begin{lemma} 
If $(\ub,y[\ub])$ is an $L^2$-local solution of $(P),$ then the associated set $\Lambda$ of multipliers is nonempty.
\end{lemma}

\section{Second order necessary conditions}\label{sec:3}

We start this section by recalling some results obtained in \cite{ABK-PartI}, the main one being the second order necessary condition of Theorem \ref{SONC}.
We then introduce the {\em Goh transform} and apply it to the quadratic form and the critical cone, {and then obtain necessary conditions on the transformed objects (see Theorem~\ref{thm:nec-sec-Goh}). We show later in Section~\ref{suf-cond-sec} that these necessary conditions can be strengthened to get sufficient conditions for optimality (see Theorem~\ref{ThmSC}).}

Let us consider an admissible trajectory  $(\ub,\yb)$.

\subsection{Assumptions \blue{on the control structure} and additional \blue{state} regularity}
\label{sonc:setting}

Consider the {\em contact sets associated to the control bounds} defined, up to null measure sets, by
$\check{I}_i  : = \{t\in [0,T];\; \ub_i(t) = \umin_i\}$, $\hat{I}_i : = \{t\in [0,T];\; \ub_i(t) = \umax_i\}$,
$I_i :=\check{I}_i  \cup \hat{I}_i$.
For $j=1,\dots,q,$ the \emph{contact set associated with the 
$j$th state constraint} is
$I^{C}_{j}:=\{ t \in [0,T];\;g_j(\yb(\cdot,t))=0\}$.
Given $0 \le a<b\le T$,
 we say that $(a,b)$ is a \emph{maximal state constrained arc}
for the $j$th state constraints, if $I^C_j$ contains
$(a,b)$ but it contains no open interval strictly containing
$(a,b)$.
We define in the same way a 
\emph{maximal (lower or upper)
control bound constraints arc}
(having in mind that the latter are defined up to a
null measure set).

We will assume the following {\em finite arc property:}
\be
\label{arcs-junc-points-hyp}
\left\{
\begin{array}{c}
\text{the contact sets for the state and bound constraints are,} \\ 
\text{{up to a finite set}, the union of finitely many maximal arcs.}
\end{array}
\right.
\ee
In the sequel we identify $\ub$ 
(defined up to a null measure set)  with a function whose 
$i$th component is constant over each interval of time that is
included, up to a zero-measure set, in either  
$\check{I}_i$ or $\hat{I}_i$.
For almost all $t\in [0,T]$, the {\em set of active constraints at time $t$}
is denoted by $(\check{B}(t), \hat{B}(t),C(t) )$ where
\be
\label{BC}
\left\{ \ba{lll}
\check{B}(t) := \{ 1\leq i \leq m;\; \ub_i(t) = \umin_i\},
\vspace{1mm} \\
\hat{B}(t) := \{1\leq i \leq m; \;  \ub_i(t) = \umax_i\},
\vspace{1mm} \\
C(t) := \{ 1\leq j \leq q; \; g_j(\yb(\cdot,t)) = 0\}.
\ea \right.\ee
These sets are well-defined over open subsets of $(0,T)$ 
where the set of active constraints
is constant, and by
\eqref{arcs-junc-points-hyp},
there exist time points called {\em junction points}
$0=:\tau_0 < \cdots < \tau_{r}:=T$,
such that the intervals $(\tau_k,\tau_{k+1})$ are {\em maximal arcs
with constant active constraints}, for $k=0,\dots,r-1.$ 
We may sometimes call them shortly {\em maximal arcs.} For $m=1$ we call junction points where a \emph{BB junction} if we have active bound constraints on both neighbouring maximal arcs, a \emph{CB junction} (resp. \emph{BC junction}) if we have a state constrained arc and an active bound constrained arc.

\begin{definition}
\label{def-bk-ck}
For $k=0,\dots,r-1,$ let  $\check{B}_k, \hat{B}_k, C_k$
denote the set of indexes of active lower and upper bound constraints, and 
state constraints, on the
maximal arc $(\tau_k,\tau_{k+1})$,
and set $B_k:= \check{B}_k \cup \hat{B}_k$.
\end{definition}

\if{As a consequence of above definitions and hypothesis \eqref{HypUad} on the admissible set of controls, we get
 the following characterization of the first order condition.
 \begin{corollary}
 \label{CorFirst}
The first order optimality condition \eqref{FirstControl} is equivalent to 
\be
\label{coco-coco}
\{t\in [0,T];\; \Psi_i^p(t) >0\} \subseteq \check{I}_i,\qquad
\{t\in [0,T];\; \Psi_i^p(t) <0\} \subseteq \hat{I}_i,
\ee
for every $(\beta,p,\dd\mu)\in \Lambda.$
\end{corollary}
}\fi

\if{
\subsubsection{On the jumps of the multiplier} Given a function $v:[0,T]\rar X$,
where $X$ is a Banach space, 
we denote (if they exist) its left and right limits 
at $\tau\in [0,T]$ by $v(\tau\pm)$,
with the convention 
$v(0-):=v(0)$, $v(T+):=v(T)$; 
then the jump of $v$ at time $\tau$ is defined as
$[v(\tau)]:=v(\tau+)-v(\tau-)$.

We denote the time derivative of the state constraints by 
\be
\label{dtgj}
g^{(1)}_j(\yb(\cdot,t)) := \frac{\dd}{\dd t} g_j(\yb(\cdot,t)) = \int_\Om c_j(x)
\dot \yb(x,t) \dd x,
\quad j=1,\ldots,q.
\ee
{Note that $g^{(1)}_j(\yb(\cdot,t))$ is an element of $L^1(0,T),$ for each $j=1,\ldots,q.$}
\begin{lemma}
\label{costate-eq-reg.l-cor}
Let $\ub$ have left and right limits at $\tau \in (0,T)$.
Then 
\be
\label{costat-eq4-c}
[\Psi^p_i(\tau)] [\ub_i(\tau)] =
[ g^{(1)}_j(\bar y(\cdot,\tau)) ][\mu_j(\tau)] = 0, 
\;\; i=1,\ldots,m, \; \;\; j=1,\ldots,q.
\ee

\end{lemma}

\begin{proof}
Since $p=p^1-\sum_{j=1}^q c_j \mu_j$, 
$p^1\in Y\subset C([0,T];H^1_0(\Om))$, $\mu \in BV(0,T)^q_{0,+},$
and any function with bounded variation has 
left and right limits, 
we have that $p(\cdot,\tau)$ has left and right limits in
$H^1_0(\Om)$ and
satisfies 
\be
[p(\cdot,\tau)] = - \sum_{j=1}^q c_j [\mu_j(\tau)],
\quad 
\text{for all $\tau\in [0,T]$.}
\ee
Consequently $\Psi^p$ has left and right limits 
over $[0,T]$, and 
\be\label{Psi-left-right}
[\Psi^p_i(\tau)] = - \sum_j [\mu_j(\tau)] \int_\Om b_i(x) c_j(x) \yb(x,t) \dd x,
\quad 
\text{for all $\tau\in [0,T]$.}
\ee
On the other hand, eliminating $\dot \yb (x,t)$ using the state equation we get that, 
for $z\in H^2(\Om) \cap H^1_0(\Om)$,
\be
\label{DG-un}
D_{u_i} g^{(1)}_j(\yb(\cdot,t)) = \int_\Om b_i(x) c_j(x) \yb(x,t) \dd x,
\ee
Combining \eqref{Psi-left-right} and \eqref{DG-un} we obtain 
\be
\label{DG-Psi}
[\Psi^p_i(\tau)] = - \sum_j [\mu_j(\tau)] D_{u_i} g^{(1)}_j(\yb(\cdot,\tau)).
\ee
Next, if $\ub$ has left and right limits at some
$\tau \in (0,T)$, then, using the state equation and \eqref{dtgj}, we get
\be\label{jump-g_j}
[ g^{(1)}_j(\yb(\cdot,\tau)) ] = 
\sum_{i=1}^m [\ub_i(\tau)] \int_\Om b_i(x) c_j(x) \yb(x,t) \dd x
= 
\sum_{i=1}^m [\ub_i(\tau)] D_{u_i} g^{(1)}_j(\yb(\cdot,\tau)).
\ee
Thus, by \eqref{DG-Psi} and \eqref{jump-g_j}, we have 
\be
\label{costat-eq4-c-f}
\sum_{i=1}^m [\Psi^p_i(\tau)] [\ub_i(\tau)] + \sum_{j=1}^q [ g^{(1)}_j(\bar y(\cdot,\tau)) ] [\mu_j(\tau)] = 0.
\ee
By the first order conditions \eqref{coco-coco} we have 
$[\Psi^p_i(\tau)] [\ub_i(\tau)] \leq 0$, for $i=1$ to $m$.
Also $[\mu_j(\tau)] \geq 0$, and if $[\mu_j(\tau)] \neq 0$, 
the corresponding state constraint
has a maximum at time $\tau$.
Then $[ g^{(1)}_j (\yb(\cdot,\tau))] \leq 0$.
So, all terms in the sums in \eqref{costat-eq4-c-f}
are nonpositive and therefore are equal to zero.
The conclusion follows.
\end{proof}

\subsection{Smoothness over an arc}
\label{arcs-smooth}

}\fi
In the discussion that follows we fix 
$k$ in $\{0,\dots,r-1\}$, and consider a maximal arc 
$(\tau_k,\tau_{k+1}),$ where the junction points.
Set $\Bb_k := \{1,\ldots,m\} \setminus B_k$
and  
\be
\label{Mij}
M_{ij}(t) := \int_\Om b_i(x) c_j(x) \yb(x,t) \dd x,
\quad 
1\leq i \leq m, \;\; 1\leq j \leq q.
\ee
Let  $\Mb_k(t)$ (of size $|\Bb_k|  \times |C_k|$)
denote the submatrix of $M(t)$ having rows with index in $\Bb_k$
and columns with index in $C_k$.

\if{
\begin{remark}
This hypothesis was already used in a different setting (i.e. higher-order state constraints in the finite dimensional case) in e.g. \cite{MR2504044,Mau79a}. Note that
condition \eqref{controllability} implies, in particular, that 
the matrix $\Mb_k(t)$ has rank $|C_k|$ over 
$(\tau_k,\tau_{k+1})$.
\end{remark}

The expression of the derivative of the $j$th state constraint, for $1\leq j \leq q$, is
\be
\label{state-cons-g}
g_j^{(1)}(\yb(\cdot,t)) = \int_\Om c_j(x)\big(f(x,t) +\Delta \yb(x,t) - \gamma \yb(x,t)^3 \big) \dd x
+ \sum_{i=1}^m M_{ij}(t) \ub_i(t),
\ee
or, in vector form, for the active state constraints (denoting by 
$g^{(1)}_{C_k}(t)$ the vector of components $g_j^{(1)}(\yb(\cdot,t))$ for $j \in C_k$), we get
\be
\label{gct-ct}
g^{(1)}_{C_k}(\yb(\cdot,t)) = G_k(t) + \Mb_k(t)^\top \ub_{\bar B_k} (t) = 0,
\ee
where $\ub_{\bar B_k}$ is the restriction of $\ub$ to the components in 
$\Bb_k$, and $G_k(t)$ takes into account the 
contributions of the integral in \eqref{state-cons-g}
and of the components of $\ub$ in $B_k$, that is, for $j\in C_k$:
\be
G_{k,j}(t) := \int_\Om c_j\big(f(x,t) +\Delta \yb(x,t) - \gamma \yb(x,t)^3 \big) \dd x
+
\sum_{i\in B_k} M_{ij}(t) \ub_i(t).
\ee
By the controllability condition \eqref{controllability}, 
$\Mb_k(t)^\top$ is onto from $ \RR^{|\bar B_k|}$ to $\RR^{|C_k|}$.
In view of the state equation, 
by an integration by parts argument,
$M(t)$ has a bounded derivative
and is therefore Lipschitz continuous.
So there exists a linear change of control variables of the form $u(t) = N_k(t) \uh(t),$  for some invertible Lipschitz continuous matrix $N_k(t)$ of size $m\times m$, 
such that, calling $\bar N_k(t)$ the upper $|\bar B_k|\times |\bar B_k|-$diagonal block of $N_k(t),$ it holds that
$\Mb_k(t)^\top \Nb_k(t)$
has its first $|C_k|$ columns being equal to the identity
matrix, the other columns having null components.
That is,  for all $\uh\in \RR^{|\Bb_k| }$:
\be
\label{mbar-n-def}
( \Mb_k(t)^\top \Nb_k(t)\uh )_j = 
\uh_j, \quad \text{for } j=1,\ldots, |C_k|.
\ee 
Over a maximal arc $(\tau_k,\tau_{k+1})$, we have that
$g^{(1)}_j(\yb(\cdot,t))=0$ for $j\in C_k$ is equivalent to 
\be
\hu_j = - G_{k,j}(t), \quad \text{for } j = 1, \ldots,|C_k|.
\ee

The following result on the regularity of the state constraint multiplier holds. Recall the definition of the switching function $\Psi^p$ given in \eqref{def-psi}.
\begin{proposition}
 There exists $a\in L^1(0,T)^m$ such that
 \begin{itemize}
  \item [(i)] 
  \be\label{dPsi}
  \dd \Psi^p (t) = a(t) \dd t - M(t) \dd \mu(t),\quad \text{on } [0,T].
  \ee
  \item[(ii)] 
We have that
$\dot \mu_{C_k}$ is locally integrable
over $(\tau_k,\tau_{k+1})$, hence $\mu_{C_k}$ is locally absolutely continuous, and the following expression holds
  \be
 \label{psi-abmp}
  0=\dot \Psi^p_{\bar B_k} (t) = a_{\bar B_k} (t) \dd t -\bar{M}_k(t) \dot \mu_{C_k}(t),\quad \text{on } (\tau_k,\tau_{k+1}).
  \ee
 \end{itemize}

\end{proposition}

\begin{proof}
By \eqref{p1} and \eqref{def-psi},
one has, for $i\in\{1,\ldots,m\}$:
\be
\label{def-psi-cons}
\Psi^p_i(t) = \alpha_i + \int_\Om b_i(x) \yb(x,t) p^1(x,t) \dd x
- \sum_{j=1}^q M_{ij}(t) \mu_j(t),
\quad i=1,\dots,m. 
\ee 
Let $a\colon (0,T) \to \cR^m$ be given by 
\be
\label{defai}
a_i(t) := \ddt \int_\Om b_i(x) \yb(x,t) p^1(x,t) \dd x
- \sum_{j=1}^q \dot M_{ij}(t) \mu_j(t),
\quad
\text{for }i=1,\ldots,m.
\ee
Note that 
$\dot M_{ij}(t) = \int_\Om b_i(x) c_j(x) \dot \yb(x,t) \dd x$
is integrable
\big(this follows integrating by parts the contribution of 
$\Delta \yb$ and since
$Y \subset C([0,T]; H^1_0(\Om))$\big), and that 
\be
\label{expr_a-mu-o}
\ddt \left( \yb p^1 \right) =  p^1 \, \Delta \yb - \yb\,\Delta p^1 + fp^1 +2\gamma \yb^3 p^1 -\beta \yb (\yb-y_d) - \sum_{j=1}^q \mu_j \yb Ac_j.
\ee
Integrating by parts the terms in \eqref{expr_a-mu-o} containing Laplacians,  we get, for the 
integral term in \eqref{defai},
\begin{equation}
\label{expr_a-mu-o-b}
\begin{aligned}
\int_\Om b_i(x) 
\ddt \left( \yb p^1 \right) \dd x &=  
 \int_\Om  b_i \left(
fp^1 +2\gamma \yb^3 p^1 -\beta
\yb (\yb-y_d) - \sum_{j=1}^q \mu_j \yb Ac_j
\right) \dd x \\
& \quad - \int_\Om  \nabla b_i (p^1\nabla \yb - \yb \nabla p^1)\dd x.
\end{aligned}
\end{equation}  
It follows that $a \in L^1(0,T)^m$ and
\eqref{dPsi} holds. Consequently $\Psi^p$ has bounded variation.

Over $(\tau_k,\tau_{k+1})$, we have 
$\dd \mu_j(t) = 0$ whenever $j\not\in C_k$, and so
\be
\label{psi-abm}
0 = \dd \Psi^p_{\Bb_k} (t) = a_{\Bb_k} (t) \dd t - \Mb_k(t) \dd \mu_{C_k}(t). 
\ee
Since $\Mb_k(t)$ is continuous and injective, and $a$
is integrable, this implies the existence of 
$\dot \mu_j(t) \in L^1(0,T)$, for $j\in C_k$. 
This yields \eqref{psi-abmp}.

\if {As expected from the finite dimensional theory 
there is no explicit contribution of the control in the 
derivative of the switching function, and we get that for
$ i\in \Bb$:
\be
 \int_\Om b_i(x) \left(
2 p \yb^2 
- y (y-y_d) -
\sum_{j=1}^q c_j \yb  \dot \mu_j(t)
\right) \dd x 
=0.
\ee
Since $\mu_j(t)$ is constant when $j\not\in C$, 
this is of the form
\be
\label{xi-mbt}
\Xi(t) + \Mb(t) \dot \mu_C(t) = 0,
\ee
with in fact $\Xi(t) = -a_{\Bb}(t)$. 
} \fi
And so, $\mu_{C_k}(t)$ is locally 
absolutely continuous.
\end{proof}

\begin{corollary}
\label{ref-reg-a}
Let the finite maximal arc property
\eqref{arcs-junc-points-hyp} 
and the uniform
controllability condition \eqref{controllability} hold.
\begin{itemize}
\item[(i)] If $f,y_d \in L^\infty(0,T;L^2(\Omega)),$ 
then $a\in L^\infty(0,T)^m.$
\item[(ii)] If additionally
$
f, y_d \in C([0,T]; L^2(\Om)),
$
then 
$\mu$ is $C^1$ over each 
maximal arc $(\tau_k,\tau_{k+1}).$
\end{itemize}
\end{corollary}

\begin{proof}
Indeed, a careful inspection of the previous
proof shows that $a$ is a sum of essentially
bounded terms, so (i) follows. 
If the additional regularity hypotheses of item~(ii) hold, then $a$ is continuous.
The regularity of $\mu$ follows from
\eqref{psi-abm}
and the local controllability assumption \eqref{controllability}. This concludes the proof.
\end{proof}

}\fi

{For the remainder of the article we make the following set of assumptions}.

\begin{hypothesis}\label{hyp-setting}
{The following conditions hold:}
\begin{itemize}

\item[1.] the  finite maximal arc property \eqref{arcs-junc-points-hyp},

\item[2.] the problem is qualified (cf. also \cite[Sec. 3.2.1]{ABK-PartI}),
i.e., for $j=1,\dots,q$:
\be
\label{qualif-P}
\left\{ \ba{lll}
\text{there exists $\eps>0$ and $u\in \Uad$ such that $v:=u-\ub$ satisfies:}
\\
g_j(\yb(\cdot,t)) + g_j'(\yb(\cdot,t))z[v](\cdot, t) <-\eps,
\text{ for all $t\in [0,T]$.} 
\ea\right.
\ee

\item[3.] 
 We assume that $|C_k| \leq |\Bb_k|,$ for $k=0,\dots,r-1,$ and that the following
{\em (uniform) local  controllability condition} holds:
\begin{equation}
\label{controllability}
\left\{
\begin{aligned}
&\text{there exists } \alpha>0, \text{ such that } | \Mb_k(t) \lambda | \geq \alpha | \lambda |,\, \text{for all } \lambda \in \RR^{|C_k|},\\
&  \text{ a.e. over each maximal arc } (\tau_k,\tau_{k+1}),  \text{ for } k=0,\dots,r-1.
\end{aligned}
\right.
\end{equation}

\item[4.] the discontinuity of the derivative of the state constraints at 
corresponding junction points, i.e., 
\be
\label{hyp-g-sharp}
\text{\it for some $c>0$: $g_j(\yb(\cdot,t)) \leq -c \dist(t, I^C_j)$, for all $t \in [0,T]$,
$j=1,\ldots,q$,}
\ee

\item[5.] the uniform distance to control bounds whenever they are not active, i.e. there exists $\delta>0$ such that, 
\be\label{hyp-geom}
\dist\big(\ub_i(t),\{\umin_i,\umax_i\}\big)\geq \delta,\quad \text{for a.a. } t\notin
{I}_i,\,\text{for all } i=1,\dots,m,
\ee
\item[6.] 
the following regularity for the data (we do not try to take the weakest hypotheses) for some $r>n+1$:
\be
\label{sonc:setting-t1-1}
\displaystyle y_0,y_{dT} \in W^{1,r}_0(\Omega)\cap W^{2,r}(\Omega),\quad 
y_d,f\in L^\infty(Q),\quad b \in W^{2,\infty}(\Om)^{m+1},
\ee
\item[7.]  
the control $\ub$ has left and right limits at the junction points $\tau_k \in (0,T)$,
 (this will allow to apply \cite[Lem. 3.8]{ABK-PartI}).
\end{itemize}
\end{hypothesis}

\begin{remark}
Hypotheses \ref{hyp-setting} 4 and 5 
are instrumental for constructing feasible perturbations of the nominal trajectory, 
used in the proof of Theorem \ref{SONC} made in \cite{ABK-PartI}.
\end{remark}

In view of point 2 above, we consider from now on $\beta=1$ and thus we omit the component $\beta$ of the multipliers.

\begin{theorem}
\label{sonc:setting-t1}
The following assertions hold.
\begin{itemize}
\item[(i)]
For any $u\in L^\infty(0,T)^m,$ the associated state $y[u]$ belongs to $C(\bar Q).$
If $u$ remains in a bounded subset of $L^\infty(0,T)^m$ then the 
corresponding states form a bounded set in
$C(\bar Q)$.
In addition, if the sequence $(u_{\ell})$
of admissible controls converges to $\ub$ a.e. on $(0,T)$,
then the associated sequence of states $(y_{\ell}:=y[u_{\ell}])$ converges
uniformly to $\yb$ in $\bar{Q}$.
\item[(ii)]  The set $\Lambda_1$ is nonempty and for every \blue{$(p,\mu)\in \Lambda_1,$} one has that $\mu\in W^{1,\infty}(0,T)^q$ and $p$  is essentially bounded in $Q$.  
\end{itemize}
 \end{theorem}

\begin{proof}We refer to \cite[Thm. 4.2]{ABK-PartI}. Note that the non-emptiness of $\Lambda_1$ follows from \eqref{qualif-P}.  \if{
(i) 
Since $b$ and the control $u$ are essentially bounded, $y=y[u]$ verifies $\dot y - \Delta y +e y +\gamma y^3=f,$
with $e(x,t) := - u(t)\cdot b(x,t)$ being essentially bounded.
In view of the regularity imposed to $y_0$ in \eqref{sonc:setting-t1-1}, we can apply 
\cite[Prop.~2.1]{bonnans:hal-00740698} and deduce that for all $q\in (2,\infty)$, the state 
$y$ belongs to $W^{2,1,q}(Q) \subset W^{1,q}(Q)$.
Taking $q>n$, it follows from the 
Sobolev Embedding Theorem (see e.g. \cite[Theorem 5, p.269]{evans}) that  $y$
is continuous (and even H\"older-continuous)
over the closure of $Q$,
with uniform bound over the set of admissible controls.
If the sequence $(u_{\ell})$
of admissible controls converges a.e. to $\ub$,
by the Dominated Convergence Theorem, 
$u_{\ell}\rar\ub$ in $L^q(0,T)$
for all $q\in [1,\infty)$. 
So, by similar arguments it can be proved that
associated sequence of states converges
uniformly to $\yb$.
\\ (ii)
By \cite[Lemma 3.4, page 82]{MR0241822} (see also \cite[p. 20]{MR712486}), given $y \in W^{2,1,q}(Q)$, its trace $y(\cdot,T)$  belongs to the {\em Besov space} $B^{2-2/q,q}(\Om)$
(defined in \cite[page 70]{MR0241822}).
Moreover, since $W^{1,q}(\Om)\subset B^{2-2/q,q}(\Om)$
with continuous inclusion, from \eqref{sonc:setting-t1-1} and latter sentence, we deduce that $p(\cdot,T)$ belongs to $B^{2-2/q,q}(\Om)$ and then, by 
\cite[Thm IV.9.1, p. 341]{MR0241822}, $p$ is in $ W^{2,1,q}(Q)$.
Hence, by an analogous reasoning of part (i), $p^1$ is also
continuous over the closure of $Q$.
The continuity of $\mu$ at junction points follows from \eqref{hyp-g-sharp} in Hypothesis \ref{hyp-setting}  and Lemma \ref{costate-eq-reg.l-cor}.
The boundedness on each arc 
of the derivative of $\mu$
follows from 
\eqref{psi-abmp} for $\dot\mu$, since by Corollary~\ref{ref-reg-a}, $a\in L^{\infty}(0,T)^m$
and by
\eqref{controllability},
$\Mb(t)$ is `uniformly injective'  over each arc.
The conclusion follows.}\fi
\end{proof}

\subsection{Second variation}

{For $(p,\mu) \in \Lambda_1,$ set}
$\kappa(x,t) := 1-6 \gamma \yb(x,t) p(x,t)$,
and consider the quadratic form
\be
\label{def-qb}
\calq[p](z,v) := \int_Q  \left(\kappa z^2 
+ 2 p \sum_{i=1}^m v_i b_i  z
\right)\dd x \dd t 
+  \int_\Om z(x,T)^2 \dd x.
\ee
Let $(u,y)$ be a trajectory, and set
\be \label{v-delta-y}
(\delta y,v) :=  (y-\yb,u-\ub).
\ee 
Recall the definition of the operator $A$ given in \eqref{lin-state-equ}.
Subtracting the state equation at 
$(\ub,\yb)$ from the one at $(u,y)$, 
we get that
\be
\label{delta-y}
\left\{
\begin{split}
&\ddt \delta y + A \delta y 
= \sum_{i=1}^m v_i b_i y -3\gamma
\yb(\delta y)^2 -
\gamma (\delta y)^3\quad \text{in } Q,\\
&\delta y=0\quad \text{on } \Sigma,\quad \delta y(\cdot,0) = 0\quad \text{in } \Omega.
\end{split}
\right.
\ee

\if{
\begin{lemma}
\label{delta-y-minus-z.l}
We have that
\be
\label{delta-y-minus-z.l-1}
\| \eta\|_Y = O \left( \| v\|_2 \| \delta y\|_Y + 
 \| \delta y\|^2_{L^4(Q)}\right).
\ee
\be
\textcolor{cyan}{\| \eta\|_Y = O \left( \| v\|_2 \| \delta y\|_{ L^{\infty}(0,T;L^2(\Om))} + 
 \| \delta y\|^2_{L^4(Q)}\right)}.
\ee

\end{lemma}

\begin{proof}
Combining theorem \ref{sonc:setting-t1}(i)
(which implies that $|\delta y|$
is uniformly bounded)
and relation \eqref{d_1-equ}, we find that 
\be
\| \eta\|_Y 
= O \left(
\| v \delta y \|_2+\| \yb (\delta y)^2 \|_2 + 
\| (\delta y)^3 \|_2 \right)
= O \left(
\| v \delta y \|_2 +\| \delta y \|^2_{L^4(Q)} \right).
\ee
By Lemma \ref{lem-uby},
\be
\| v \delta y \|_2 
\leq
\| v\|_2 \| \delta y\|_{L^\infty(0,T;L^2(\Om))}  
\leq
\| v\|_2 \| \delta y\|_Y.
\ee
The conclusion follows. 
\end{proof}
}\fi
\if{
\begin{lemma}
 The weak solution of \eqref{d_1-equ} coincides with a mild solution in $C(0,T;L^2(\Om))$.
\end{lemma}
\begin{proof}
 The statement can be easily derived from \cite{Ball:1977} taken into account that here we have a zero order term with coefficient depending on space and time.
\end{proof}
}\fi
\if{\begin{proof}
 Let $\underline{r}(t):=r(t,\cdot)$ and $\underline{\tilde r}(t):=\tilde r(t,\cdot)$. Both functions belong to $L^1(0,T;L^2(\Om))$. We consider the very weak formulation of \eqref{d_1-equ}:
 $\eta(0)= 0$ and for any 
$\phi \in \dom(-\Delta)$,  the function $t\mapsto \la \phi , \eta(t)\ra$ 
is absolutely continuous over $[0,T]$ and satisfies
\begin{align}
\ddt\la \phi , \eta(t)\ra+\la -\Delta \phi , \eta(t)\ra 
= \la \phi , \underline{\tilde r}(t) + \underline{r}(t)\eta(t)\ra,
\;\; \text{for a.a. $t\in [0,T]$.}
\end{align}
 Define $g(t):=\underline{\tilde r}(t) + \underline{r}(t)\eta(t)$ and $\psi(t):=\int_0^t e^{(t-s)\Delta}g(s) \dd s$. Then, by Ball \cite{Ball:1977} $\psi$ is the unique very weak solution of
 \begin{align}
\dot \psi(t)  -\Delta \psi(t)= g(t),
\;\; \text{for a.a. $t\in [0,T],\quad \psi(0)=0$}
\end{align}
which is given by 
\begin{align}
\ddt\la \psi , \eta(t)\ra+\la -\Delta \psi , \eta(t)\ra 
= \la \psi , g(t)\ra,
\;\; \text{for a.a. $t\in [0,T]$.}
\end{align}
By the uniqueness of the very weak solution we obtain $\psi=\eta$ which concludes the proof.
\end{proof}}\fi

See the definition of the Lagrangian function $\call$ given in equation \eqref{Lagrangian} of the Appendix.
\begin{proposition}
\label{prop:expansion}
Let $(p,\mu)\in \Lambda_1$, and let
 $(u,y)$ be a trajectory. 
Then 
\begin{multline}
\label{DeltaLformula}
\call [p,\mu](u,y,p) - \call
[p,\mu](\ub,\yb,p)\\
 = \int_0^T  \Psi^p(t) \cdot v(t) \dd t + \half \calq[p](\delta y,v) - 
\gamma
\int_Q  p (\delta y)^3 \dd x\dd t.
\end{multline}
Here, we omit the dependence of the Lagrangian on $(\beta,p_0)$ being equal to $(1,p(\cdot,0))$.
\end{proposition}
\begin{proof}
 We refer to \cite[Prop. 4.3]{ABK-PartI}.
\end{proof}

\if{
\begin{proof}
Use $\Delta \call$ to denote the l.h.s. of \eqref{DeltaLformula}. We have
\be
\label{DeltaL}
\begin{aligned}
 \Delta & \call =\, J(u,y)-  J(\ub,\yb)  +\int_Q p \left( -\frac{\dd}{\dd t}\delta y +\Delta \delta y- \gamma (y^3-\yb^3) \right)\dd x \dd t\\
 & +\int_Q p\left(\sum_{i=1}^m v_ib_iy+\sum_{i=0}^m \ub_i b_i\delta y  \right)\dd x \dd t  
+\sum_{j=1}^q \int_0^T \int_\Om c_j \delta y \, \dd x \dd \mu_j(t)
\\ =&    \int_Q \delta y\left( \half\delta y+\yb-y_d \right) \dd x\dd t+
   \int_\Omega \delta y (x,T)\Big( \half\delta y(x,T)+\yb(x,T)-y_{dT}(x) \Big) \dd x\\
&+\sum_{i=1}^m\alpha_i\int_0^T v_i \dd t+ \int_Q p\left( -\frac{\dd}{\dd t}\delta y+\Delta \delta y-\gamma(\delta y^3 +3\yb \delta y^2+3\yb^2 \delta y)\right)\dd x\dd t
 \\  & + \int_Q p \left( \sum_{i=1}^m v_ib_iy+\sum_{i=0}^m \ub_i b_i\delta y  \right)+\sum_{j=1}^q \int_0^T \int_\Om c_j \delta y \, \dd x \dd \mu_j(t)
.\end{aligned}
\ee
By \eqref{costate-eq} we obtain
\be
\label{pdotdeltay}
\begin{split}
\int_Q p \frac{\dd}{\dd t}\delta y\, \dd x\dd t = &  -\int_Q pA\,\delta y\,\dd x\dd t + \sum_{j=1}^q \int_0^T \int_\Om c_j \delta y \, \dd x \dd \mu_j(t)\\
&+    \int_Q \delta y\left(\yb-y_d \right) \dd x\dd t+    \int_\Omega \delta y (x,T)\left( \yb(x,T)-y_{dT}(x) \right) \dd x.
\end{split}
\ee
Thus, from \eqref{DeltaL} and \eqref{pdotdeltay} we get
\begin{multline}
\Delta \call = \half    \int_Q \delta y^2 \dd x\dd t+
\half    \int_\Omega \delta y (\cdot,T)^2  \dd x + \sum_{i=1}^m\alpha_i\int_0^T v_i \dd t\\
+ \int_Q p\left( -\gamma [\delta y^3 +3\yb \delta y^2]+  \sum_{i=1}^m v_ib_iy \right)\dd x\dd t,
\end{multline}
which leads to \eqref{DeltaLformula} in view of the definition of $\Psi^p_i$ given in \eqref{def-psi}. This concludes the proof.
\end{proof}}\fi

\subsection{Critical directions}

Recall the definitions of $\check{I}_i,\hat{I}_i$ and $I^C_j$ given in Section   \ref{sonc:setting}, and remember that we use $z[v]$ to denote the solution of the linearized state equation \eqref{lineq2}  associated to $v.$

We define the {\em cone of critical directions}
at $\ub$ in $L^2$, or in short {\em critical cone,} by
\be
C:=
\left\{
\begin{split}
& (z[v],v)\in Y\times L^2(0,T)^m;\\
& v_i(t)\Psi_i^p(t)=0\, \text{ a.e. on } [0,T],\, \text{for all } (p,\mu)\in \Lambda_1 \\ 
&v_i(t)\geq 0
\text{ a.e. on }  \check{I}_i,\,
 v_i(t)\leq 0\, \text{ a.e. on } \hat{I}_i,\, \text{ for } i=1,\dots,m, \\
& \int_\Om c_j(x) z[v](x,t)\dd x\leq 0 \text{ on } I^C_j, \, \text{ for } j=1,\dots,q
\end{split}
\right\}.
\ee
The {\em strict critical cone} is defined below, and it is obtained by imposing 
that the linearization of active constraints is zero,
\be
\label{Cs}
C_{\rm s}:=
\left\{
\begin{split}
& (z[v],v)\in Y\times L^2(0,T)^m;\; v_i(t)=0\, 
\text{ a.e. on }  I_i, \text{ for } i=1,\dots,m, 
\\
&
\int_\Om c_j(x) z[v](x,t)\dd x  = 0\, \text{ on }I^C_j, \text{ for }  j=1,\dots,q
\end{split}
\right\}.
\ee
Hence, clearly $C_{\rm s} \subseteq C,$ and $C_{\rm s}$ is a closed subspace of $Y\times\Uspace.$
\if{Now, note that in the interior of each $I^C_j$ one has, for every $(z[v],v)\in C_{\rm s},$
\be
\begin{aligned}
0&=\ddt \big( g'_j(\yb(\cdot,t)) z[v](\cdot,t) \big) =  \ddt  \int_\Om c_j(x) z[v](x,t) \dd x  \\
&= \int_\Om c_j(x)  \dot z[v](x,t)  \dd x 
=\int_\Om c_j(x)  \Big({-{(Az[v])}(x,t)} + (v(t)\cdot b(x)) \yb(x,t) \Big) \dd x,
\end{aligned}
\ee
which can be rewritten as
\be
\label{dotgz}
\sum_{i=1}^m v_i(t) M_{ij}(t) 
= \int_\Om c_j(x) {(Az[v])}(x,t) \dd x,
\ee
in view of the definition of $M_{ij}$ given in \eqref{Mij}.
Therefore,  over any arc $(a,b)$ we have $g'_j(\yb(\cdot,t)) z[v](\cdot,t) =0$
for $t\in (a,b)$ if and only if $g'_j(\yb(\cdot,a)) z[v](\cdot,a) =0$
and \eqref{dotgz} holds over~$(a,b)$. 
We define the {\em entry (resp. exit) point} of a time interval
$(t',t'')$ as $t'$ (resp. $t''$). 
{This induces the consideration of the following sets}
\begin{equation*}
C_{\rm e} := \left\{ 
\begin{split}
&(z[v],v)\in Y\times L^2(0,T)^m;\;\\
&g'_j(\yb(\cdot,\tau_k)) z[v](\tau_k)  =0,\, \text{if } j\in C_k,\, \text{for } k=0,\dots,r-1
\end{split}
\right\},
\end{equation*}
\begin{equation*}
C_{\rm n}:=
\left\{
\begin{split}
& (z[v],v)\in Y\times L^2(0,T)^m;\; v_i(t)=0 \,
\text{ a.e. on }  I_i,\;\text{for } i=1,\dots,m,\\
&
\sum_{i=1}^m v_i(t) M_{ij}(t) 
= \int_\Om c_j(x) {(Az[v])}(x,t) \dd x\text{ a.e. on }I^C_j, \text{for } j=1,\dots,q
\end{split}
\right\}.
\end{equation*}
With these definitions, we can write the strict critical cone as
$
C_{\rm s} = C_{\rm e} \cap C_{\rm n},
$
and prove the following result.
}\fi

\if{
\begin{lemma}
\label{cs-cap-inf.l}
{$C_{\rm s} \cap \Big(Y\times L^\infty(0,T)^m\Big)$
is dense in  $C_{\rm s},$ with respect to the $Y\times L^2(0,T)^m$-topology}.
\end{lemma}

\begin{proof}
In view of Dmitruk's density Lemma (see \cite[Lemma 1]{Dmi08}),
it is enough to prove that $C_{\rm n} \cap \Big( Y\times L^\infty(0,T)^m\Big)$
is a dense subset of $C_{\rm n}$. 

Let us then take $(z,v)\in C_{\rm n}.$
Recall the definition of the junction times $\tau_k$ given after equation \eqref{Mij}.  Fix $k\in\{0,\dots,r-1\}.$
Note that we can take a partition of $[0,T]$, say
$0=t_0\leq \dots \leq t_{\ell} \leq \dots \leq  t_N=T$, 
such that $(t_{\ell},t_{\ell+1})$ is contained in some 
$(\tau_k,\tau_{k+1})$, and on $(t_{\ell},t_{\ell+1})$ a fixed set of the rows of
$M(t)$ is linearly independent with rank equal to the one
of $M(t)$. Now consider the matrix $\bar M_k$ given after \eqref{Mij}.
Using the same notation as in
\eqref{gct-ct}, let us write 
$v_{\bar B_k}$ to refer to the restriction
of $v$ to the components in $\bar B_k.$ 
For each $t\in (t_{\ell},t_{\ell+1})$, we can write
\be
v_{\bar B_k}(t) = v_{\bar B_k,0}(t) +v_{\bar B_k,1}(t),
\ee
where $v_{\bar B_k,0}(t)\in \Ker \bar{M}_k(t)^\top$ and $v_{\bar
  B_k,1}(t)\in \Im \bar{M}_k(t)$ for almost all $t,$  hence 
$v_{\bar  B_k,1}(t)= \bar{M}_k(t) \lambda_{k}(t)$ for some $\lambda_{k}(t) \in \cR^{|C_k|}.$
Let
$E_{C_k}(t)$ be the $|C_k|$-dimensional vector with components
\be
E_{C_k,j}(t):=\int_\Om c_j(x) {(Az)}(x,t) \dd x, \quad j\in C_k. 
\ee
 Then \eqref{dotgz}
can be rewritten as 
\be
E_{C_k}(t) = \bar{M}_k(t)^\top v_{\bar B_k}(t) = 
\bar{M}_k(t)^\top v_{\bar  B_k,1}(t)
=
\bar{M}_k(t)^\top \bar{M}_k(t) \lambda_k(t), 
\ee
and, therefore, $\lambda_k(t) = 
\big(\bar{M}_k(t) ^\top\bar{M}_k(t) \big)^{-1} E_{C_k}(t),$
so that
\begin{align} 
\label{v-prime-1}
v_{\bar B_k,1}(t) = 
\bar{M}_k(t) \lambda_k(t)
=
\bar{M}_k(t)
\big(\bar{M}_k (t)^\top\bar{M}_k(t) \big)^{-1}
E_{C_k}(t).
\end{align}
By an integration by parts (in space) argument, it follows
that $E_{C_k}(t)$ is a continuous function, and so is
$\bar{M}_k(t)$. Therefore,
$v_{\bar B_k,1}$ is continuous on each maximal arc.
We may also view the application 
$z\mapsto v_{\bar B_k,1}$ 
as a linear and continuous mapping say
\be
L_1: \;\; Y \rar \prod_{k=0}^{r-1} \Lip(\tau_k,\tau_{k+1})^{|C_k|}
\ee
where $C_k$ is the set of active state constraints on 
$(\tau_k,\tau_{k+1})$ and, for $t'<t''$,
$\Lip(t',t'')$ is the Banach space of
continuous real functions with domain $(t',t'')$,
endowed with the norm
\be
\|f\|_{\Lip(t',t'')} :=
\sup_{t\in (t',t'')} |f(t)| +
\sup_{t,\tau\in (t',t'')} 
\frac{|f(t)-f(\tau)| }{|t-\tau|},
\ee
with the convention ``$0/0=0$''.

For any $\eps>0,$  there exists 
$v_{\bar B_k,0}^\eps$ in $L^\infty(0,T)^{|B_k|}$ such that
$\|v_{\bar B_k,0}^\eps-v_{\bar B_k,0}\|_2<\eps$, 
it has zero components for indexes corresponding to active
control bound constraints, and 
$v_{\bar B_k,0}^\eps (t)\in\Ker \bar{M}_k(t)^\top$ for a.a. $t.$
In fact, to construct this $v_{\bar B_k,0}^\eps$ it suffices to project 
an approximation of $v_{\bar B_k,0}$
obtained by a truncation argument on the kernel $\Ker \bar{M}_k(t)^\top$.
In what follows we shall abuse notation and use the same symbol to denote a vector and its canonical immersion in $\cR^m.$
Let $z_\eps$ be the unique solution in $Y$ of the linearized equation 
$$
\dot z_\eps + A z_\eps  = \sum_{i=1}^m (L_1({z_\eps}) + v_{\bar B,0}^\eps +v_B)_i  b_i\, \yb,
$$ 
with the usual initial and boundary conditions, and where $v_B$ is the restriction of $v$ to the set $B.$ 
Set  $v_{\bar B,1}^{\eps}:=L_1(z_\eps),$  $v_{\bar B_k}^{\eps}:=v_{\bar B_k,1}^{\eps}+v_{\bar B_k,0}^\eps,$ and define $v_\eps$ to have the restriction to $\bar B_k$ equal to $v_{\bar B_k}^{\eps}$ and the restriction to $B_k$ equal to $v.$
Then $v_\eps$ is in $C_{\rm n}\cap (Y\times L^{\infty}(0,T)^m)$ and $\|v_\eps-v\|_2=O(\eps).$ Hence, $C_{\rm n} \cap \Big( Y\times L^\infty(0,T)^m\Big)$
is a dense subset of $C_{\rm n}$.
The conclusion follows.
\end{proof}

\subsubsection{Radiality of critical directions}

According to Aronna {\em et al.} \cite[Definition 6]{MR3555384}, a critical direction $(z,v)$ is {\em quasi radial} if there exists $\tau_0>0$ such that, for $\tau\in [0,\tau_0],$ the following conditions are satisfied:
\begin{gather}
\label{max-g-o2}
\max_{t\in [0,T]} \left\{ g_j(\yb(\cdot,t)) + \tau g_j'(\yb(\cdot,t)) z(t)\right\}  = o(\tau^2),\quad \text{for } j=1,\dots,q,\\
\label{qr2}
\umin_i \leq \ub_i(t) +\tau v_i(t) \leq \umax_i,\quad \text{a.e. on } [0,T],\quad \text{for } i=1,\dots,m.
\end{gather}

\begin{lemma}
\label{hyp-g-sharp.l}
Every direction in $C_{\rm s}\cap \Big( Y\times L^\infty(0,T)^m \Big)$ is quasi radial.
\end{lemma}

\begin{proof}
Let $(z,v) \in C_{\rm s}\cap \Big( Y\times L^\infty(0,T)^m \Big).$ 
Then \eqref{qr2}  follows from \eqref{hyp-geom}.
Let us next prove \eqref{max-g-o2}.
The function $h(t) := g'_j(\yb(t))z(t)$ has
the derivative $\dot h(t) = \int_\Om c_j(x) \dot z(x,t) \dd x$,
so that 
$|\dot h(t) | \leq \|c_j\|_{L^2(\Om)} \|\dot z(\cdot,t)\|_{L^2(\Om)}$ and hence, $\dot h 
        \in L^2(0,T)$. 
Let $0\leq t' < t'' \leq T$. 
By the Cauchy-Schwarz inequality,
for any $\eps>0$, if $t''-t'$ is small enough: 
\be
| h(t'')-h(t') | \leq \int_{t'}^{t''} |\dot h(t)| \dd t 
\leq \sqrt{t''-t'} \| \dot h\|_{L^2(t',t'')}. 
\ee
Let $(a,b)$ be a maximal constrained arc with say $a>0$.
Take $t'<a$,  and $t''=a$.
When $t\uparrow a$, by the Dominated Convergence Theorem, 
$\|\dot{ h}\|_{L^2(t',t'')} \rar 0$.
Given $\eps>0$, we deduce with \eqref{hyp-g-sharp}
that for $\tau>0$ and $t<a$ close enough to $a$:
\be
\label{hyp-g-sharp.l-1}
g_j(\yb(\cdot,t)) + \tau g_j'(\yb(\cdot,t)) z(t) \leq -c(a-t) + \tau \eps \sqrt{a-t},    
\quad \text{for all $t\in (t',a)$.}
\ee
The maximum of the r.h.s. of \eqref{hyp-g-sharp.l-1} over
$t\in [a-\eps,a]$
is attained when
\be
c \sqrt{a-t} = \half \tau\eps,
\qquad
a-t = \frac{\tau^2\eps^2}{4c^2}.
\ee
So the r.h.s. of \eqref{hyp-g-sharp.l-1}
is less or equal than $\tau^2\eps^2/(4c)$. 
Since we can take $\eps$ arbitrarily small, $\tau^2\eps^2/(4c)$ is not greater than $o(\tau^2)$. 
For $t>b$ we have a similar result.
The conclusion follows.
\end{proof}

Combining the previous result with Lemma
\ref{cs-cap-inf.l}, we deduce that:

\begin{corollary}
\label{qrd-dense-coro}
The set of quasi radial critical directions of $C_{\rm s}$ is dense in $C_{\rm s}.$
\end{corollary}

\if{ 
\subsubsection{Application of Lemma \ref{hyp-g-sharp.l}}
Let the state constraint have the more general form
\be
g_j(y(\cdot,t)) = \int_\Om \gamma_j(y(x,t)) \dd x
\ee
We look for conditions on $\gamma_j$
for which $h(t) := g'(\yb(t))z(t)$ has a derivative in $L^2(0,T)$,
so that the Lemma \ref{hyp-g-sharp.l} applies. 
Formally
\be
h_j(t) = g'_j(y)z = \int_\Om \gamma'_j(y(x,t)) z(x,t) \dd x
\ee
and
\be
\dot h_j = \int_\Om \left( 
\gamma''_j(y) \dot y z 
+
\gamma'_j(y) \dot z\right)\dd x
\ee
Since $\dot y$ and $\dot z$ belong to $L^2(Q)$, 
we need that $\gamma''_j(y) = 0$ and that
$\gamma'_j(y) $ is essentially bounded.
Therefore it is natural to assume 
\eqref{stateconstraint}, with  $c_j$ in $L^\infty(\Om)$ for each $j.$

\begin{proposition}
\label{derg'}
 If $c_j\in L^\infty(\Omega)$ for all $j=1,\dots,\ell,$ then $t\mapsto g'_j(\yb(t))z(t)$ has derivative in $L^2(0,T).$
\end{proposition}
} \fi 

}\fi

\subsection{Second order necessary condition}
We recall from \cite[Thm. 4.7]{ABK-PartI}.
\begin{theorem}[Second order necessary condition]
\label{SONC}
Let the admissible trajectory $(\ub,\yb)$ be an $L^\infty$-local solution of $(P)$. Then
\be
 \max_{ (p,\mu) \in \Lambda_1 } \calq[p] (z,v)\geq 0, 
\qquad \text{for all } (z,v) \in C_{\rm s}.
\ee
\end{theorem}
\if{\begin{proof}We refer to .
\if{
 Let $(z,v)\in C_{\rm s}.$ 
By Corollary \ref{qrd-dense-coro},
there exists a sequence $(z_{\ell},v_{\ell})$ of quasi radial directions converging
to
$v$ in $Y\times L^2(0,T)^m$.
 Doing as in 
\cite[Theorem 8]{MR3555384}, 
we get the existence of a multiplier 
$(p_{\ell},\dd\mu_{\ell}) \in \Lambda_1$ (with $\Lambda_1$ defined in Section~\ref{sec:red-prob}), such that
 \be
 \calq[p_{\ell},\dd\mu_{\ell}](z_\ell,v_{\ell}) \geq 0.
 \ee
By Lemma \ref{multipliers}, $\Lambda_1$ 
is bounded and, extracting if necessary a subsequence, we may
assume that
$(p_{\ell},\dd\mu_{\ell})$  weakly-$*$ converges to some 
\blue{$(p,\mu)\in \Lambda_1$}. Since $\calq$ is 
weakly-$*$ continuous with respect to the multipliers,
and $(z_\ell,v_{\ell})\rar (z[v],v)$ in $Y\times L^2(0,T)^m$, 
we easily deduce that 
$\calq[p_{\ell},\dd\mu_{\ell}](z_\ell,v_{\ell}) \rar
\calq[p,\dd\mu](z[v],v).$
The conclusion follows.
}\fi
\end{proof}
}\fi

\subsection{Goh transform} 
Given a critical direction $(z[v],v)$, set 
\be
\label{Goh}
w(t) := \int_0^t v(s) \dd s; \quad
B(x,t) := \yb(x,t)  b(x);
\quad
\zeta(x,t) = z(x,t)  - B(x,t) \cdot w(t).
\ee
Then $\zeta$ satisfies the initial and
boundary conditions
\be
\begin{aligned}
\label{boundaryzeta}
\zeta(x,0)=0  \text{ for } x\in \Omega,\quad
\zeta(x,t) =0\text{ for } (x,t)\in \Sigma.
\end{aligned}
\ee
Remembering the definition \eqref{lin-state-equ}
of the operator $A$, we obtain that 
\be
\dot \zeta + A \zeta = 
\left(\dot z+ A z - \sum_{i=1}^m v_i B_i \right) 
- \sum_{i=1}^m w_i ( A B_i +\dot B_i),\quad \zeta(\cdot ,0)=0,\quad
\zeta(x,t) =0\text{ on } \Sigma.
\ee
In view of the linearized state equation
\eqref{lineq2}, the term between the
large parentheses  in the latter equation vanishes. Since 
$\dot B_i = b_i \dot \yb$ 
it follows that 
\be
\label{zeta}
\dot \zeta(x,t)  
+ {(A\zeta)} (x,t) = B^1(x,t) \cdot w(t),\quad \zeta(\cdot ,0)=0,\quad
\zeta(x,t) =0\text{ on } \Sigma, 
\ee 
where 
\be
\label{B1}
B^1_i := -f b_i +2\nabla\yb \cdot \nabla b_i + \yb \Delta b_i - 2\gamma\yb^3b_i,\quad \text{for }i=1,\dots ,m.
\ee
{Equation \eqref{zeta} is well-posed 
since $b\in W^{2,\infty}(\Omega)$, and the  solution $\zeta$ belongs to $Y$.} We use $\zeta[w]$ to denote the solution of \eqref{zeta} corresponding to $w.$

\subsection{Goh transform of the quadratic form} 
Recall that $(\ub,\yb)$ is a feasible trajectory.
Let $\pb=p[\ub]$ be the costate associated to $\ub,$ and set
\be
W:=Y\times L^2(0,T)^m \times \RR^m.
\ee
Let $S(t)$ be the time dependent
symmetric $m\times m-$matrix with generic term
\be\label{def:S}
S_{ij}(t) := \int_\Om b_i(x) b_j(x) p(x,t) \yb(x,t) \dd x,
\quad \text{for }1 \leq i,j\leq m.
\ee 
Set
\be\label{chi}
\chi := \ddt (p\yb)
= pf+
p \Delta \yb - \yb \Delta  p
+ 2 p \yb^3
- \yb(\yb-y_d)
- \yb \sum_{j=1}^q c_j \dot \mu_j.
\ee
Observe that
\be
\dot S_{ij}(t) =
\int_\Om b_i b_j 
\ddt (p\yb) \dd x
= 
\int_\Om b_i b_j 
\chi \dd x.
\ee
Since $\yb$, $p$ 
belong to $L^\infty(0,T,H^1_0(\Om))$,
and 
$y_d$, $\yb^3$, $\dot \mu$
are essentially bounded, 
integrating by parts the terms in \eqref{chi} involving the Laplacian
operator  and using \eqref{sonc:setting-t1-1}, we obtain that 
$\dot S_{ij}$ is essentially bounded.
So we can define 
the continuous quadratic form on $W:$ 
\be
\label{def-hatq}
\whq[p,\mu] (\zeta,w,h) :=  \int_0^T \hat q_I(t) \dd t + \hat q_T,
\ee
where
\begin{multline}
\label{def-hatq1}
\hat q_I :=  \disp 
\int_\Om \kappa \left(\zeta+ \yb\sum_{i=1}^m b_i w_i \right)^2  \dd x
-w^\top \dot S w\\
 -2 \sum_{i=1}^m  w_i \int_\Om  
\Big[ \zeta \big( -\Delta b_i p - 2\nabla b_i\cdot \nabla p+b_i(\yb-y_d)
+ b_i\sum_{j=1}^q c_j \dot\mu_j \big) - p B^1\cdot w\Big]\dd x,
\end{multline}
and
\begin{multline}
\label{def-hatq2}
\hat q_T:=  \\
\int_\Om \left[ \Big(\zeta(x,T)+\yb(x,T) \sum_{i=1}^m h_i b_i(x)\Big)^2  + 2 \sum_{i=1}^m h_i  b_i (x) p(x,T) \zeta (x,T) \right]\dd x
 +
h^\top  S(T) h.
\end{multline}

\begin{lemma}[Transformed second variation]\label{lem:transf-sec-var}
For $v\in L^2(0,T)^m,$ and $w\in AC([0,T])^m$ given by the Goh transform \eqref{Goh}, {and for all $(p,\mu)\in \Lambda_1$,} we have
\be
\calq[p](z[v],v) = \whq[p,\mu] (\zeta[w], w , w(T)).
\ee
\end{lemma}

\begin{proof}
We first replace 
$z$ by $\zeta + B\cdot w = \zeta + \yb \sum_{i=1}^m w_i b_i$ {in $\calq$}, and define
\begin{multline}
\label{wtq}
{\wtq}:= \int_Q \Big[ \kappa( \zeta + \yb \sum_{i=1}^m w_i b_i )^2
+2p\sum_{i=1}^m v_ib_i(\zeta + \yb \sum_{j=1}^m w_j b_j)\Big]\dd x\dd t\\
+\int_\Omega\big(\zeta(T) + \yb(T) \sum_{i=1}^m w_i(T) b_i\big)^2\dd x.
\end{multline}
We aim at proving that $\wtq$ coincides with $\whq.$ This will be done by removing the dependence on $v$ from the above expression.  
For this, we have to deal with the bilinear term {in $\wtq$}, namely with
\be
\label{hatcalqb}
\wtq_b :=\wtq_{b,1} + {2\sum_{i=1}^m {\wtq_{b,2i}},}
\ee
where, omitting the dependence on the multipliers for the sake of simplicity of the presentation, 
\be
\label{hatcalqb12}
{\wtq_{b,1}}:= 2
\int_0^T v^\top S  w  \dd t\quad \text{and }\quad
{{\wtq_{b,2i}} := \disp
\int_0^T v_i \int_\Om b_i  p \zeta  \dd x  \dd t,\,\,\text{for } i=1,\dots,m.}
\ee
Concerning $\wtq_{b,1}$,
since $S$ is symmetric, we get, integrating by parts,
\be
\label{wtqb1}
{\wtq_{b,1}} =\left[w^\top S w\right]_0^T - \int_0^T w^\top \dot S w \dd t.
\ee
Hence $\wtq_{b,1}$ is a function of $w$ and $w(T)$.
Concerning $\wtq_{b,2i}$ defined in \eqref{hatcalqb12}, integrating by parts, we get
\be
\label{Qi''2}
\wtq_{2,bi}
=
w_i(T) \int_\Om b_i p(x,T) \zeta (x,T) \dd x 
-
\int_0^T w_i \int_\Om  b_i  \ddt \big( p\zeta  \big)\dd x \dd t .
\ee
For the derivative inside the latter integral, one has
\be\label{dtpzeta}
\ddt \big( p(x,t)\zeta (x,t) \big) = 
-\Delta p \zeta + p\Delta \zeta-\zeta\left(
(\yb-y_d)+\sum_{j=1}^q c_j \dot\mu_j\right)+p B^1\cdot w.
\ee
By Green's Formula:
\be\label{intDeltapzeta}
\int_Q w_ib_i\big(-\Delta p \zeta + p\Delta \zeta\big)\dd x\dd t =
\int_Q w_i \big( \Delta b_i p +2\nabla b_i\cdot \nabla p) \zeta \dd x\dd t.
\ee
Using \eqref{dtpzeta} and \eqref{intDeltapzeta} in {the expression} \eqref{Qi''2} yields
\begin{multline}
\label{wtqb2i}
\wtq_{b,2i}=  
w_i(T) \int_\Om b_i p(x,T) \zeta (x,T) \dd x 
+\int_Q w_i \Big[ \zeta \big( -\Delta b_i p - 2\nabla b_i\cdot \nabla p\\
 + b_i(\yb-y_d) + b_i\sum_{j=1}^q c_j \dot\mu_j \big) - p B^1\cdot w\Big]\dd x\dd t.
\end{multline}
Hence, $\wtq_{b,2}$ is a function of $(\zeta,w,w(T))$. Finally, 
{putting together \eqref{wtq}, \eqref{hatcalqb}, \eqref{wtqb1} and \eqref{wtqb2i} yields an expression for $\wtq$ that does not depend on $v$ and coincides with $\whq$ (in view of its definition given in \eqref{def-hatq}-\eqref{def-hatq2}). This concludes the proof.}
\end{proof}

\begin{remark}
{The matrix appearing as coefficient of the quadratic term $w$ in $\whq$ (see \eqref{def-hatq1}) is the symmetric $m\times m$ time dependent matrix  $R(t)$ with entries}
\be
\label{matrixR}
R_{ij} := \int_\Om \left(\kappa b_i b_j \yb^2 - \dot S_{ij}
+ p (b_i B^1_j + b_j B^1_i)  \right) \dd x,\quad \text{for } i,j=1,\dots,m.
\ee
\end{remark}

\subsection{Goh transform of the critical cone}
Here, {we apply the Goh transform to the critical cone and obtain the cone $PC$ in the new variables $(\zeta,w,w(T))$. We then define its closure $PC_2,$ that will be used in the next section to prove second order sufficient conditions. In Proposition \ref{cc-m1-p}, we characterize $PC_2$ in the case of scalar control.}
\subsubsection{Primitives of strict critical directions}
Define the set of primitives of strict critical directions as
\be\label{PC_2}
PC := \left\{
\begin{split}
 &(\zeta,w,w(T))\in  Y\times H^1(0,T)^m\times \RR^m;\\ 
 & (\zeta,w) \text{ is given by \eqref{Goh}  for some } (z,v)\in C_{\rm s}
\end{split}
\right\},
\ee
{which is obtained by applying the Goh transform \eqref{Goh} to $C_{\rm s}$,} and let
\be
PC_2 := 
\text{closure of $PC$ in $Y\times L^2(0,T)^m\times\RR^m$}.
\ee
Remember Definition \ref{def-bk-ck}
of the active constraints sets $\check{B}_k,$ $\hat{B}_k$, $B_k=\check{B}_k \cup \hat{B}_k,$ $C_k$.
\begin{lemma}
For any $(\zeta,w,h)\in PC$, it holds
\be
\label{w-B}
w_{B_k} (t) =
\frac{1}{\tau_{k+1}  - \tau_{k}}
 \int_{\tau_{k}}^{\tau_{k+1}} w_{B_k} (s) \dd s,\quad \text{for } k=0,\dots,r-1.
 \ee
\end{lemma}

\begin{proof}
Immediate from the constancy of $w_{B_k}$ a.e. on each $(\tau_k,\tau_{k+1}),$ for any $(\zeta,w,h)\in PC.$
\end{proof}

Take $(z,v)\in C_{\rm s},$ and let $w$ and $\zeta[w]$ be given by the Goh transform \eqref{Goh}. Let $k\in \{0,\dots,r-1\}$ and take an
index $j \in C_k.$ Then
$\ds 0=\int_\Om c_j(x) z(x,t) \dd x$ on $(\tau_{k} , \tau_{k+1})$.
Therefore,
letting $M_j(t)$ denote the $j$th column of the matrix 
$M(t)$ (defined in \eqref{Mij}), one has
\be
\label{w-C}
M_j(t) \cdot w(t) = - \int_\Om c_j(x)  \zeta[w](x,t) \dd x,\quad \text{on } (\tau_k,\tau_{k+1}),\, \text{for } j\in C_k.
\ee
We can rewrite \eqref{w-B}-\eqref{w-C} 
in the form
\be
\label{Aw=Bw}
\cala^k(t) w(t) = \left( \calb^k  w \right) (t),
\quad
\text{on } (\tau_{k} , \tau_{k+1}),
\ee
where $\cala^k(t)$ is an $m_k\times m$ matrix with $m_k := |B_k| + |C_k|,$ and $\calb^k \colon L^2(0,T)^m \to H^1(\tau_k,\tau_{k+1})^{m_k}.$
We can actually consider 
$\calb:=(\calb^1,\ldots,\calb^{r})$
as a linear continuous mapping from
$L^2(0,T)^m$ to $\ds\Pi_{k=0}^{r-1}  H^1(\tau_k,\tau_{k+1})^{m_k},$
and $\cala:=(\cala^1,\ldots,\cala^{r})$
as a linear continuous mapping from 
$L^2(0,T)^m$ into 
$\ds\Pi_{k=0}^{r-1} L^2(\tau_k,\tau_{k+1})^{m_k}$.
For each $t \in (\tau_k,\tau_{k+1})$, let us use $\cala(t)$ to denote the matrix $\cala^k(t).$ 
We have that, for a.e. $t\in (0,T),$  $\cala(t)$ is  of maximal rank, 
so that there exists a unique measurable
$\lambda(t)$ 
(whose dimension is the rank of $\cala(t)$ and depends on $t$)
such that 
\be
\label{wdecomposition}
w(t) = w_0(t) + \cala(t)^\top\lambda(t),
\quad
\text{with } w_0(t) \in \Ker \cala(t).
\ee
Observe that
 $\cala(t) \cala(t)^\top$ has a continuous time
derivative and is uniformly invertible
on $[0,T].$ So, 
$(\cala(t) \cala(t)^\top)^{-1}$
is linear continuous from $H^1$
into $H^1$ (with appropriate dimensions)
over each arc,
and 
$\cala(t) \cala(t)^\top \lambda (t) = (\calb w)(t)$
a.e.
We deduce that
$t\mapsto (\lambda (t), w_0(t))$
belongs to $H^1$
over each arc $(\tau_k,\tau_{k+1}).$
So, in the subspace $\Ker (\cala-\calb)$,
$w\mapsto\lambda(w)$
is linear continuous, {considering the $L^2(0,T)^m$-topology in the departure set, and the $\ds\Pi_{k=0}^{r-1} H^1(\tau_k,\tau_{k+1})^{m_k}$-topology in the arrival set}.
Since $(\cala-\calb)$ is linear
continuous over $L^2(0,T)^m$
we have that
\be
\label{pc2-ker-a-b}
w\in  \Ker (\cala-\calb),\qquad \text{for all } (\zeta,w,h)\in PC_2.
\ee
While the inclusion {induced by} \eqref{pc2-ker-a-b}
could be strict, we see that for any $(\zeta,w,h)\in PC_2$,
$\lambda(w)$ and $\cala w$ are well-defined in 
the $H^1$ spaces, and the following initial-final conditions hold:
\be
\label{lem-pc2-1}
\left\{ \ba{lll}
{\rm (i)} & \disp
\text{$w_i=0$\, a.e. on $(0,\tau_1),$ for each $i\in B_0,$}
\vspace{1mm} \\ {\rm (ii)} & \disp
  \text{$w_i=h_i$\, a.e. on $(\tau_{r-1},T),$ for each $i\in B_{r-1},$}
\vspace{1mm} \\ {\rm (iii)} & \disp
\text{$g_j'(\yb(\cdot,T))[\zeta(\cdot,T)+B(\cdot,T)\cdot h ]=0$ if $j\in C_{r-1}$.}
\ea\right. 
\ee
From the definitions of $C_s$ (see \eqref{Cs}) and of $PC_2$, we can obtain  additional continuity conditions at the {\em bang-bang} junction points:
\be
\text{if $i\in B_{k-1} \cup B_k$, then $w_i$ is continuous at  $\tau_k$,\quad for all $(\zeta,w,h)\in PC_2.$}
\ee

\begin{remark}
Another example is when $m=1,$ the state constraint is active for $t<\tau$ and the control constraint is active for $t>\tau$, then $w$ is continuous at time $\tau$.
This is similar to the ODE case studied in \cite[Remark 5]{MR3555384}.
\end{remark}

We have seen that {over each arc $(\tau_k,\tau_{k+1})$,} 
$\lambda^k:=\lambda(w)$ is pointwise well-defined, and it possesses right limit at the entry point and left limit at the exit point,
denoted by 
$\lambda(\tau_k^+)$ and $\lambda(\tau_{k+1}^-)$, respectively.
Let {$c_{k+1}\in\RR^m$}
be such that, for some
$\nu^{k+i}$,
\be
\label{LemmaJ-1}
{c_{k+1}}  = \cala^{k+i}(\tau_{k+1})^\top\nu^{k+i}, \;\;\text{for } i=0,1,
\ee 
meaning that $c_{k+1}$ is a linear combination of the 
rows of $\cala^{k+i}(\tau_{k+1})$ for both $i=0,1.$

\begin{lemma}
\label{LemmaJ}
{Let $k=0,\dots,r-1,$ and let $c_{k+1}$} satisfy \eqref{LemmaJ-1}.
Then, the junction condition
\be
\label{LemmaJ-2}
{c_{k+1}\cdot \big(w(\tau_{k+1}^+)- w(\tau_{k+1}^-) \big)} = 0,
\ee
holds for all $(\zeta,w,h) \in PC_2$.
\end{lemma}

\begin{proof}
Let $(\zeta,w,h)$ in $PC$, {and set $c:=c_{k+1}$ and $\tau:=\tau_{k+1}$ in order to simplify the notation}.
Then
\be
\label{LemmaJ-2-a}
c\cdot w(\tau) 
 = 
(\nu^k)^\top \cala^k(\tau) w(\tau)
= 
(\nu^k)^\top \cala^k(\tau) (\cala^k(\tau))^\top \lambda^k(\tau).
\ee
By the same relations for index $k+1$ we conclude that
\be
\label{nukeq}
(\nu^{k})^\top \cala^{k}(\tau) (\cala^{k}(\tau))^\top \lambda^{k}(\tau)
=
(\nu^{k+1})^\top \cala^{k+1}(\tau) (\cala^{k+1}(\tau))^\top \lambda^{k+1}(\tau).
\ee
Now let $(\zeta,w,h)\in PC_2$. Passing to the limit in the above relation \eqref{nukeq}
written for $(\zeta[w_\ell],w_{\ell},h_{\ell})\in PC$, $w_{\ell}\rar w$ in $L^2(0,T)^m$, {$h_{\ell}\rar h$}
(which is possible since
$\lambda(t)$ is uniformly Lipschitz over each arc),
we get that  \eqref{nukeq} holds for any 
$(\zeta,w,h)\in PC_2$, from which the conclusion follows.
\end{proof}

By {\em junction conditions} at the junction time $\tau = \tau_k\in (0,T)$,
we mean any relation of type \eqref{LemmaJ-2}. 
Set 
\be\label{PC-set}
PC'_2 := \{ (\zeta[w],w,h);\, w\in \Ker (\cala-\calb),
\text{\eqref{LemmaJ-2} holds, for all $c$
satisfying \eqref{LemmaJ-1}}\}.
\ee
We have proved that
\be
\label{pc2incl-pcp2}
PC_2 \subseteq PC'_2.
\ee
In the case of a scalar control ($m=1$) we can show that these
two sets coincide. 
\subsubsection{Scalar control case}
The following holds: 

\begin{proposition}
\label{cc-m1-p}
If the control is scalar, then
\be
\label{cc-m1-p1}
PC_2 = \left\{ 
\ba{lll}
(\zeta[w],w,h)\in Y\times L^2(0,T)\times \RR; \;\;
w \in \Ker (\cala-\calb); 
\vspace{1mm} \\
\text{$w$ is continuous at BB, BC, CB junctions} 
\vspace{1mm} \\
\text{$\lim_{t\downarrow 0}w(t)=0$ 
if the first arc is not singular}
\vspace{1mm} \\
\text{$\lim_{t\uparrow T}w(t)=h$ 
if the last arc is not singular}
\ea \right\}.
\ee
\end{proposition} 
For a proof we refer to \cite[Prop. 4 and Thm. 3]{MR3555384}.

\if{
\subsubsection{Extension: under construction}

Set for $t\in (0,T)$:
\be
X^k_t := (\Ker \cala^k(t) )^\perp = \range( \cala^k(t)^\top).
\ee
For $t$ close enough to 
$(\tau_k, \tau_{k+1})$, this subspace has dimension $m_k$. 
Set, for $t$ close to $\tau= \tau_{k}$:
\be
Y^k_t:=X^k_{t} \cap X^{k-1}_{t}; \;\;
W^k_t:= X^k_{t-1} \cap (Y^k_t)^\perp  \;\;
Z^k_t := X^{k}_{t} \cap (Y^k_t)^\perp;
F^k_t = ( X^{k-1}_{t} + X^{k}_{t} )^\perp.
\ee
We need to assume the nondegeneracy hypothesis
\be
\label{hyp-ND}
\text{$Y^k_t$ has constant dimension for $t$ close to $\tau_{k+1}$.} 
\ee
Assume in the sequel that $t$ close to $\tau_{k+1}$.
Then we have the direct sums 
\be
\label{sum-dir-ywzf}
\RR^m = Y^k_t   \oplus W^k_t  \oplus Z^k_t  \oplus F^k_t.
\ee
These four spaces have constant dimension, and 
the projection of $w_t\in PC_2$ over each of them
is continuous. 

\begin{proposition}
Let \eqref{hyp-ND} hold. 
Then $PC_2 = PC'_2(\ub)$.
\end{proposition}

\begin{proof}
We need to prove the converse inclusion to 
\eqref{pc2incl-pcp2}. So, let $w \in PC'_2(\ub)$.
We need to construct some $w'\in PC$
such that $\|w'-w\|_2$ is arbitrarily small.

First regularization by convolution,
$\rho_\eps(t) = \eps^{-1}\rho(t/\eps)$ with $rho$
$C^\infty$ with support in unit ball, nonnegative with
unit integral, and 
$w_\eps := w * \rho_\eps$. Then
\be
\| \dot w_\eps \|_\infty 
\leq \|w\|_1 \| \rho_\eps \|_\infty
\leq O( \eps^{-1} ).
\ee

XXX

Given $\eps>0$, let $w^1\in \Uad$ be of class
$C^\infty$ and such that $\|w^1-w\|_2< \eps$.
We next define another control $w^2$ by induction over the arcs.
Let $k\in\{0,\ldots, r-1\}$ be such that $w^2$ and the associated $\zeta^2$
have been defined over $[0,\tau_k)$.
We start with $k=0$. 
Over the arc $(\tau_k, \tau_{k+1})$ we define $w^2(t)$ as the 
closest element to $w^1(t)$ while satisfying that satisfies the 
bound and state constraints in the definition of $PC$:
\be
w^2(t) \in \argmin_{v\in\RR^m}\{ \half |v-w^1(t)|^2; \; 
\cala^k(t) v = \left( \calb^k  w^2 \right) (t)
\}.
\ee
In view of the definition of these linear constraints, we can write
(where $w^b$ corresponds to bound constraints)
\be
w^2(t) = w^b(t) + A(t) (w^2(t),zeta(t))
\ee
This definition makes sense: for each $t\in (\tau_k, \tau_{k+1})$,
$\left( \calb^k  w^2 \right) (t)$ since the values corresponding
to bound values is fixed and the contribution of the state contraint
depends only on  $\zeta^2(t)$ (solution of the equation in $\zeta$
associated with $w^2$).
(which depends only of the past values of $w^2$). 
We easily obtain that for some Lipschitz function $\Xi$:
\be
w^2(t) = \Xi(w^1(t), \zeta^2(t))
\ee
I think that we can deduce that 
\be
\| w^2 - w \|_2 = O(\|w^1-w\|_2) = O(\eps).
\ee
It follows that 
\be
\| \zeta^2 - \zeta \|_2  = O(\eps).
\ee
In addition over each arc $w^2$ is of class $C^\infty$
with uniformly bounded derivatives (of any order).
However $w^2$ has jumps at junction points due to 
the way it is defined. But these jumps are also of order $\eps$. 
By the infinite dimensional 
\cite[Thm. 2.200]{MR1756264} 
version of Hoffman's Lemma \cite{MR0051275}, 
applied to the space $PC$ (endowed with the norm of
$H^1_0(0,T)^m$), there exists 
$w^3 \in PC$ such that 
\be
\| w^3 - w \|_2\leq = \| w^3 - w^2 \|_2  + \| w^2 - w \|_2 
= O(\eps).
\ee
we may compute 
{\bf XXX TO BE REVISED}
\end{proof}
}\fi

\subsection{Necessary conditions after Goh transform}
\if{We define in the usual way (cf. Ambrosio \cite{MR1079985}) the space
$
BV(0,T;L^2(\Om)).
$}\fi
{The following second order necessary condition in the new variables follows.}
\begin{theorem}[Second order necessary condition]\label{thm:nec-sec-Goh}
If $(\ub,\yb)$ is an $L^\infty$-local solution of problem \eqref{P}, then
\be
 \max_{ (p,\mu) \in \Lambda_1 } \whq[p,\mu] (\zeta,w,h) \geq 0,\quad \text{on } PC_2.
 \ee
\end{theorem}

\begin{proof} 
    Let $(\zeta,w,h) \in PC_2$.
    Then there exists a sequence $(\zeta_\ell:=\zeta[w_\ell],w_{\ell},w_{\ell}(T))$ in $PC$ with
\be
\label{strong-cv}
(\zeta_\ell,w_{\ell},w_{\ell}(T)) \rar (\zeta,w,h),\quad \text{in } Y \times L^2(0,T) \times \RR.
\ee
Let $(z_\ell,v_\ell)$ denote, for each $\ell,$ the corresponding critical direction in $C_{\rm s}.$
By  Lemma \ref{lem:transf-sec-var} and Theorem \ref{SONC},
there exists
$(p_{\ell},\mu_{\ell}) \in \Lambda_1$
such that 
\be
0 \leq \calq [p_{\ell}](z_\ell,v_{\ell}) = \widehat \calq [p_{\ell},\mu_{\ell}](\zeta_\ell,w_{\ell},h_{\ell}).
\ee
We have that $(\dot\mu_\ell)$ is  bounded in $L^\infty(0,T)$
(this is an easy variant of \cite[Cor. 3.12]{ABK-PartI}).
Extracting if necessary a subsequence, we may assume that
$(\dot\mu_{\ell})$ weak* converges in $L^\infty(0,T)$ to some $\dd\mu$. Consequently, the corresponding solutions $p_\ell$ of \eqref{equationp2} weakly converge to $p$ in $Y$, $p$ being the  costate associated with $\mu$.
In view of the definition of $\widehat\calq$
in \eqref{def-hatq}, we will show that, by strong/weak convergence,
\be
\lim_{\ell\to \infty} \widehat\calq [p_{\ell},\mu_{\ell}](\zeta_\ell,w_{\ell},h_{\ell})
=
\lim_{\ell\to \infty} \widehat\calq [p_{\ell},\mu_{\ell}](\zeta,w,h)
=
\widehat\calq [p,\mu](\zeta,w,h).
\ee
\blue{The first equality is an easy consequence of \eqref{strong-cv}
combined with the boundedness of $p_{\ell},\mu_{\ell}$.
We next discuss the second equality.
For the terms having integral in time it is enough to detail the most
delicate term that is the
contribution of $\Delta p$ to~$\dot S$. 
 Denote by $S_\ell$ the matrix $S$ in \eqref{def:S}
for $p$ equal to $p_\ell$.
Since $w$ belongs to $L^2(0,T)^m$, it is enough to show that
$\dot S_\ell$ weakly* converges in $L^\infty(0,T)^{m \times m}$. 
Again we detail the contribution of the most delicate term in $\dot S_\ell$,
namely for all $1\leq i,j \leq m$, 
$\int_\Om b_i b_j \yb \Delta p_\ell$,
and it is enough to check that it weakly* converges in $L^\infty(0,T)$
to $\int_\Om b_i b_j \yb \Delta p$.
} 

\blue{
Integrating by parts in space, we see
that we only need to check that 
$\nu_\ell := \int_\Om b_i b_j \nabla \yb \cdot \nabla p_\ell$
 weakly* converges in $L^\infty(0,T)$ to
 $\nu := \int_\Om b_i b_j \nabla \yb \cdot \nabla p$. That is,
 $\int_0^T \nu_\ell(t) \varphi(t) \dd t \rar \int_0^T \nu(t) \varphi(t) \dd t$,
 for all $\varphi \in L^1(0,T)$.
But since $\nu_\ell$ is bounded in $L^\infty(0,T)$  (using $H^{2,1}(Q)\subset L^{\infty}(0,T;H^1_0(\Om))$)
say of norm less than $M>0$,
it is enough to take the test functions $\varphi$ in $L^\infty(0,T)$
instead of $L^1(0,T)$. Indeed assume that
\be
\label{nu-varphi}
\int_0^T \nu_\ell (t) \varphi(t) \dd t \rar \int_0^T \nu (t) \varphi(t) \dd t,\quad \text{for all } \varphi \in L^{\infty}(0,T).
\ee
Then, let $\varphi\in L^1(0,T)$.
Given $\eps>0$, there exists $\varphi_\eps$ in $L^\infty(0,T)$ such that
$\| \varphi_\eps - \varphi \|_1 < \eps$. Then
\be 
\limsup_{\ell \to\infty} \int_0^T \nu_\ell(t)\varphi(t) \dd t \leq 
\int_0^T \nu(t) \varphi_\eps(t)\dd  t + M \eps
\leq \int_0^T \nu(t) \varphi(t)\dd  t + 2 M \eps.
\ee
So it suffices to prove \eqref{nu-varphi}.
Let $\varphi$ in $L^{\infty}(0,T)$.
Since (extracting if necessary a subsequence)
$p_\ell$ weakly converges to $p$ in 
$H^{2,1}(Q)$ we have that $\nabla p_\ell$
weakly converges to $\nabla p$ in $L^2(Q)$, hence \eqref{nu-varphi}  easily follows.
}

\blue{
On the other hand, for the contribution of the final time it is enough to observe that
$p_\ell(x,T)$  does not depend on $\ell$. 
}
\end{proof}


\section{Second order sufficient conditions}\label{suf-cond-sec} 

In this section we derive second order sufficient optimality
conditions.

\begin{definition}
We say that an $L^2$-local solution $(\ub,\yb)$ satisfies
the {\em weak quadratic growth condition} if 
 there exist $\rho >0$ and $\varepsilon>0$ such that,
\begin{equation}
\label{growth-cond}
F(u)- F(\ub) \geq \rho ( \|w\|^2_2 + |w(T)|^2),
\end{equation}
where $(u,y[u])$ is an admissible trajectory, $\|u-\ub\|_2<\varepsilon,$  $v := u-\ub$, and
$w(t):=\int_0^t v(s)\dd s$.
\end{definition}

\begin{remark}
Note that \eqref{growth-cond} is a quadratic growth condition
in the $L^2$-norm of the perturbations $(w,w(T))$ obtained after Goh transform.
\end{remark}

The main result of this part is given in Theorem \ref{ThmSC} and establishes sufficient conditions for a trajectory to be a $L^2$-local solution with weak quadratic growth.

Throughout the section we assume Hypothesis \ref{hyp-setting}.
In particular, we have by Theorem \ref{sonc:setting-t1} that the state and costate are essentially
bounded.

Consider the condition
\be
\label{jumpcond}
g_j'( \yb(\cdot,T))( \zetab(\cdot,T) + B(\cdot,T)\hb) = 0,\, \text{ if } T \in I^C_j \text{ and } [\mu_j(T )] > 0,\,\, \text{for } j=1,\dots,q.
 \ee
We define
\be
\label{p-cal}
PC_2^*:=
\left\{
\begin{split}
& (\zeta[w],w, h) \in Y\times L^2(0,T)^m \times \RR^m; \;
w_{B_k} \text{ is constant on each arc;} \\
& \eqref{zeta}, \eqref{w-C},\eqref{lem-pc2-1}\text{(i)-(ii)},\eqref{jumpcond}\text{ hold.}
\end{split}
\right\}.
\ee
Note that $PC_2^*$ is a superset of $PC_2$.

\if{
We define 
\be
\begin{aligned}
\calp_2^{(0)} &:= \left\{
\begin{array}{l}
 (w, h, \zeta) \in L^2(0,T)^m \times \RR^m \times Y; \;
 \begin{array}{l}
 \text{$w_{B_k}$ is constant on each arc}
 \end{array}
 \end{array}
 \right\},\\
\calp_2^{(1)} &:= \left\{
 (w, h, \zeta) \in \calp_2^{(0)}
; \;
  \eqref{w-C}, \eqref{zeta},\eqref{lem-pc2-1}\text{(i)-(ii)}\text{ hold}
 \right\},\\
\calp_2^{(2)}&:= \left\{(y, h, \xi) \in  \calp_2^{(1)} ;\; \eqref{lem-pc2-1}\text{(iii)}\text{ holds}\right\},
\end{aligned}
\ee
and set
\be\label{p-cal}
\calp_2:=
\left\{
\begin{array}{ll}
\calp^{(2)}_2, & \text{ if }T \in C\text{ and }[\mu(T )] > 0,\text{ for some }(\beta, p, \dd \mu) \in \Lambda,\\
\calp^{(1)}_2,& \text{otherwise}.
\end{array}
\right.
\ee
}\fi

\begin{definition} Let $W$ be a Banach space. 
We say that a quadratic form $Q \colon W \rar \RR$ is a \emph{Legendre
form} if it is weakly lower semicontinuous, positively homogeneous of degree $2$,
i.e., $Q(tx) = t^2 Q(x)$ for all $x \in  W$ and $t > 0$, and such that if $x_{\ell} \rightharpoonup  x$ and $Q(x_{\ell}) \rar Q(x)$, then $x_{\ell} \rar x$.

\if{We say that a quadratic mapping $Q \colon W \rar \RR$
is a Legendre form if it is 
sequentially weakly lower semi continuous 
and, if
$w_{\ell} \rar w$ weakly in $W$ and $Q(w_{\ell} ) \rar Q(w)$, then $w_{\ell} \rar w$ strongly.}\fi
\end{definition}

We assume, in the remainder of the article, the following strict complementarity conditions for the control and the state constraints.
\begin{hypothesis}\label{hyp2}
The following conditions hold:
\begin{align}
\label{strict-compl}
\left\{
\begin{array}{ll} 
\text{(i)}& \text{for all } i=1,\dots,m:\\
&\hspace{-5mm} \ds\max_{(p,\mu)\in \Lambda_1} \Psi^p_i(t) > 0
 \text{ in the interior of }
                                                         \check{I}_i,\,
                                                          \text{at }
                                                          t=0 \text{
                                                          if }
                                                          0\in\check{I}_i,
                                                          \text{ at }
                                                          t=T \text{
                                                          if }
                                                          T\in\check{I}_i,
  \\
   & \hspace{-5mm} \ds
      \min_{(p,\mu)\in \Lambda_1} \Psi^p_i(t) < 0
       \text{ in the interior of } \hat{I}_i,\, \text{at } t=0 \text{
       if } 0\in \hat{I}_i, \text{ at } t=T \text{ if }
       T\in\hat{I}_i,
  \\
\text{(ii)} & {\text{there exists $(p,\mu)\in \Lambda_1$ such that
              $\supp \dd \mu_j = I^C_j$, for all $j=1,\dots,q$.}
              }
\end{array}
\right.
\end{align}
\end{hypothesis}

\begin{theorem}\label{ThmSC}
Let Hypotheses \ref{hyp-setting} and \ref{hyp2} be satisfied. Then the following assertions hold.
\begin{itemize}
\item[a)]
Assume that 
  \begin{itemize}
  \item[(i)] $(\ub,\yb)$ {is a feasible trajectory with nonempty associated set of multipliers $\Lambda_1$;}
  \item[(ii)] {for each $(p,\mu) \in \Lambda_1,$
      $\whq[p,\mu](\cdot)$} is a Legendre form on the space
    \\
    $\{(\zeta[w],w, h) \in Y\times L^2(0,T)^m \times \RR^m\}$; and 
  \item[(iii)]  the {\em uniform positivity} holds, i.e. there exists $\rho>0$ such that
  \be\label{Q-coerc_cond}
\max_{ (p,\mu) \in \Lambda_1 } \whq[p,\mu](\zeta[w],w,h) \ge \rho ( \|w\|^2_2 + |h|^2), 
\;\; \text{for all $(w,h) \in  PC_2^*$.}
  \ee
 \end{itemize}
 Then $(\ub,\yb)$ is a $L^2$-local solution satisfying 
the weak quadratic growth condition.
\item[b)]
Conversely, for an admissible trajectory $(\ub,y[\ub])$ satisfying the growth condition \eqref{growth-cond}, it holds
 \be\label{Q-coerc_cond2}
 \max_{ (p,\mu) \in \Lambda_1 } \whq[p,\mu](\zeta[w],w,h) \ge \rho ( \|w\|^2_2 + |h|^2), 
\quad \text{for all $(w,h) \in  PC_2$.} 
  \ee
\end{itemize}
 \end{theorem}

The remainder of this section is devoted to the proof of Theorem \ref{ThmSC}. We first state some technical results.

\subsection{A refined expansion of the Lagrangian}\label{suf-cond-sec-ref} 

Combining with the linearized state equation \eqref{lineq2},
we deduce that $\eta$ given by
$\eta:=\delta y -z$,
satisfies the equation
\begin{equation}
\label{d_1-equ}
\left\{
\begin{aligned}
&\dot \eta -\Delta \eta 
= r \eta + \tilde r \quad \text{in } Q,\\
&\eta=0 \quad \text{on } \Sigma,\quad \eta (\cdot,0) = 0\quad \text{in } \Omega
\end{aligned}
\right.
\end{equation}
where $r$ and $\tilde r$ are defined as
\begin{align}
 r:=-3\gamma \yb^2 + \sum_{i=0}^m \ub_i b_i,\qquad  \tilde r :=\sum_{i=1}^m v_i b_i \delta y -3\gamma \yb(\delta y)^2- \gamma (\delta y)^3.
\end{align}

{Let $(\ub,\yb)$ be an admissible trajectory.} We start with a refinement of the expansion of the Lagrangian of Proposition \ref{prop:expansion}.

\begin{lemma}
\label{suf-cond.p-7-lem} 
Let $(u,y)$ be a trajectory,
$(\delta y,v):=(u-\ub,y-\yb)$,
$z$ be the solution of the linearized state equation 
\eqref{lineq2}, and 
$(w,\zeta)$ be given by the Goh transform \eqref{Goh}.
Then
\be
\label{suf-cond.p-7-lem-eq} 
\begin{split}
{\rm (i)} \,\,& \disp \| z \|_{{L^2(Q)}} + \|z(\cdot,T)\|_{L^2(\Omega)} = O( \|w\|_2 + |w(T)|),\\
{\rm (ii.a)} \,\,& \disp \|\delta y \|_{L^2(Q)} + \|\delta y(\cdot,T) \|_{L^2(\Omega)}  = O( \|w\|_2 + |w(T)|),\\
{\rm (ii.b)} \,\, & \|\delta y\|_{L^\infty(0,T;H_0^1(\Omega))} = O(\|w\|_\infty),\\  
{\rm (iii)}\,\, & \disp \|\eta \|_{L^\infty(0,T;L^2(\Omega))} + \| \eta(\cdot,T)\|_{L^2(\Omega)} = o ( \|w\|_2 + |w(T)| ).
\end{split}
\ee
\end{lemma}

Before doing the proof of Lemma \ref{suf-cond.p-7-lem}, let us recall the following property:
\begin{proposition}
\label{PropMild}
The equation
\be
\label{eqPsi}
\dot\Phi-\Delta\Phi+a\Phi = \hat f,\quad \Phi(x,0)=0,
\ee
with 
$a\in L^\infty(Q)$, $\hat f\in L^1(0,T;L^2(\Omega)),$
and homogeneous Dirichlet conditions on $\partial\Om\times (0,T)$,
has a unique solution $\Phi$ in $C([0,T];L^2(\Omega))$,
that satisfies
\be
\| \Phi \|_{C([0,T];L^2(\Omega))} \leq c\|\hat f\|_{L^1(0,T;L^2(\Omega))}.
\ee
\end{proposition}
\begin{proof}
{This follows from the estimate for mild solutions in the
  semigroup theory, see e.g. \cite[Theorem 2]{MR3767765PlusErratum}.}
\end{proof}


\begin{proof}[Proof of Lemma \ref{suf-cond.p-7-lem}]
(i) Since $\zeta$ is solution of \eqref{zeta}, it satisfies \eqref{eqPsi} with
\be
\label{eqahatf}
a:=-3\gamma\yb^2+\sum_{i=0}^m \ub_ib_i,\quad
\hat f:=\sum_{i=1}^m w_i B^1_i,
\ee
where $B^1_i$ is given in \eqref{B1}.
One can see, in view of Hypothesis \ref{hyp-setting}, that $\hat f \in L^1(0,T;L^2(\Omega))$ since the terms in brackets in \eqref{eqahatf} belong to $L^\infty(0,T;L^2(\Omega)).$
Thus, from Proposition~\ref{PropMild} we get that $\zeta \in C([0,T];L^2(\Omega))$ and
\be
\label{suf-cond.p-6-lem-1}
\|\zeta\|_{L^\infty(0,T;L^2(\Omega))} = O(\|\hat f\|_{L^1(0,T;L^2(\Omega))}) = O( \|w\|_{1} ).
\ee
Thus, due to Goh transform \eqref{Goh} and Lemma \ref{state-equ-estimA.l}, we get that $z$ belongs to $C([0,T];L^2(\Omega))$ and we obtain the estimate {\rm (i).}

We next prove the estimate (ii) for $\delta y.$ Set $\zeta_{\delta y} :=\delta y - (w\cdot b)\yb.$ Then
\be
\label{dotzetadeltay}
\dot{\zeta}_{\delta y} - \Delta \zeta_{\delta y} + a_{\delta y}\zeta_{\delta y} = \hat{f}_{\delta y},
\ee
with
\be
\begin{aligned}
a_{\delta y}&:= 3\gamma\yb^2+3\gamma\yb\zeta_{\delta y}+\gamma(\zeta_{\delta y})^2-(\ub\cdot b),\\
\hat f _{\delta y}&:=  \sum_{i=1}^m w_i \left[ \yb \Delta b_i + \nabla b_i \cdot \nabla \yb - b_i(2\gamma\yb^3+f)\right].
\end{aligned}
\ee
By Theorem \ref{sonc:setting-t1}, $\zeta_{\delta y}$ is in $L^\infty(Q)$, hence $a_{\delta y}$ is essentially bounded.
Furthermore, in view of the regularity  Hypothesis \ref{hyp-setting} and Lemma \ref{state-equ-estimA.l}, $\hat f_{\delta y} \in L^1(0,T;L^2(\Omega))$.
We then get, using Proposition \ref{PropMild}, 
\be
\label{zetadeltayest}
\|\zeta_{\delta y}\|_{L^\infty(0,T;H_0^1(\Omega))} \leq O(\|w\|_1).
\ee
From the latter equation and the definition of $\zeta_{\delta y}$ we deduce {\rm (ii.a)}.
{Since 
\be
\nabla(\delta y) = \nabla (\zeta_{\delta y}) + \sum_{i=1}^m w_i (\yb\nabla b_i + b_i\nabla\yb),
\ee
applying the $L^\infty(0,T;L^2(\Omega))$-norm to both sides, and using \eqref{zetadeltayest} and Lemma \ref{state-equ-estimA.l} we get {\rm (ii.b).}}

The estimate in {\rm (iii)} follows from the following consideration.
To apply Proposition \ref{PropMild} to equation \eqref{d_1-equ} we easily verify that $r$ is in $L^{\infty}(Q)$ 
and $\tilde r$ in $L^1(0,T;L^2(\Om))$. Consequently, we have
\begin{equation}
 \begin{aligned}
 &\norm{\eta}{C([0,T];L^2(\Om))}   \le c \norm{\sum_{i=1}^m v_i b_i \delta y -3\gamma 
 \yb(\delta y)^2- \gamma (\delta y)^3}{L^1(0,T;L^2(\Om))}  \\
 & \quad\le \norm{v}{2} \norm{b}{\infty} \norm{\delta y}{2}
 +3\gamma 
 \norm{\yb}{\infty}\norm{(\delta y)^2}{L^1(0,T;L^2(\Om))}  
 +\gamma \norm{(\delta y)^3}{L^1(0,T;L^2(\Om))}.
\end{aligned}
\end{equation}
Now, since  $\norm{v}{2} \rar 0$ and
$\|\delta y\|_{\infty}\to 0$ (by similar arguments to those of the proof of (i) in Theorem~\ref{sonc:setting-t1}), we get {\rm (iii).}

\end{proof}

\begin{proposition}
\label{suf-cond.p}
\blue{Let $(p,\mu) \in \Lambda_1.$}
Let $(u_\ell)\subset \Uad$ and let us write $y_\ell$ for the corresponding states. Set $v_\ell:=u_\ell-\ub$ and assume that $v_{\ell} \to 0$ a.e.  
Then, 
\begin{multline}
\label{suf-cond.p-1}
\call[p,\mu](\ub+v_{\ell},y_\ell) = \call[p,\mu](\ub,\yb) \\+ \int_0^T \Psi^p(t)\cdot v_{\ell}(t) \dd t
+ \half
\whq[p,\mu](\zeta_{\ell},w_{\ell},w_{\ell}(T))+ o( \|w_{\ell}\|_2^2 + |w_{\ell}(T)|^2),
\end{multline}
where $w_{\ell}$ and $\zeta_{\ell}$ are given by the Goh transform \eqref{Goh}.
\end{proposition}

\begin{proof} 
Since $(v_{\ell})$ is bounded in $L^{\infty}(0,T)^m$ and 
converges a.e. to 0, 
it converges to zero in any $L^p(0,T)^m$.
For simplicity of notation we omit the index $\ell$ for the remainder of the proof.
Set $\delta y:= y[\ub+v]-\yb.$ By Proposition \ref{prop:expansion}
it is enough to prove that 
\begin{gather}
\label{suf-cond.p-2-1}
 \left|\calq[p](\delta y,v) - \whq[p,\mu](w,w(T),\zeta) 
\right| =   o( \|w\|^2_2 + |w(T)|^2),\\
\label{suf-cond.p-2-2}
\left| \int_Q p (\delta y)^3 \right| = o( \|w\|^2_2 + |w(T)|^2).
\end{gather}
We have, setting as before $\eta:= \delta y -z$ where $z:=z[v],$
\be
\label{suf-cond.p-3}
\begin{aligned}
& \calq[p](\delta y,v) - \whq[p,\mu](\zeta,w,w(T))   =  
\calq[p](\delta y,v) - \calq[p](z,v)
\vspace{1mm} \\ &
= 
2 \int_Q  (v\cdot b) p \eta \dd x\dd t +
\int_Q\kappa (\delta y + z) \eta\dd
x \dd t
\vspace{1mm} 
+\int_{\Om} (\delta y(x,T) 
+ z(x,T))\eta(x,T)   \dd x,
\end{aligned}
\ee
and therefore, since 
the state and costate are essentially bounded:
\be
\label{suf-cond.p-4}
\begin{aligned}
\big| \calq[p](\delta y,v) - &\whq[p,\mu](\zeta,w,w(T))\big | 
\leq 2
\left| \int_Q  (v\cdot b)p \eta \dd x\dd t \right| +
O( \|\delta y + z \|_2\|\eta\|_2)\\
& \quad +
O( \|(\delta y + z)(\cdot,T)\|_{L^2(\Om)} \|\eta(\cdot,T)\|_{L^2(\Om)}).
\end{aligned}
\ee
In view of lemma \ref{suf-cond.p-7-lem}, the
  `big O' terms in the r.h.s. are of the desired order and it remains
  to
  deal with the integral term.
We have, integrating by parts in time,
\be\label{DeltaQeq}
\int_Q (v\cdot b) p \eta \dd x\dd t
=  \int_\Om \big(w(T)\cdot b(x)\big) p(x,T) \eta(x,T) \dd x
-
\int_Q (w\cdot b) \ddt (p \eta) \dd x\dd t. 
\ee
For the first term in the r.h.s. of \eqref{DeltaQeq} we get, in view of 
\eqref{suf-cond.p-7-lem-eq}(ii), 
\begin{multline}
\left|
\int_\Om (w(T)\cdot b(x)) p(x,T) \eta(x,T) \dd x
\right| \\
=
O( | w(T)| \|\eta(\cdot,T)\|_{L^2(\Omega)}) = o( \|w\|^2_2 + | w(T)|^2).
\end{multline}
And, for the second term in the r.h.s. of
\eqref{DeltaQeq}, since $b$
is essentially bounded,
\if{
\be\label{DeltaQest}
\left|
\int_Q (w\cdot b) \ddt (p \eta) \dd x\dd t
\right| 
=
O\left( \|w\|_{2} \left\| \ddt (p \eta)  \right\|_{L^2(Q)}
\right).
\ee
} \fi 
and $p$ and $\eta$ satisfy 
\eqref{equationp2} and \eqref{d_1-equ},
respectively, we have that, 
\be
\label{intdtpeta}
\begin{aligned}
&\hspace{4cm}\ddt (p \eta)  =
\varphi_0 + \varphi_1 + \varphi_2,\quad \\
&\varphi_0 := p \Delta \eta- \eta \Delta p; \;
\varphi_1 :=  ( v\cdot b )  p \delta y; \;
\varphi_2 := p e (\delta y)^2 -\eta \left( 
y - y_d + \sum_{j=1}^q c_j \dot \mu_j(t)
\right).
\end{aligned}
\ee
{\em Contribution of $\varphi_2.$}
Since $y$, $p$ and $\dot \mu$ are essentially bounded (see Theorem \ref{sonc:setting-t1}), we get
 \be
\left|
\int_Q (w\cdot b) \varphi_2 \right| = O\left(\|w(\delta y)^2+w\eta\|_2\right) = o(\|w\|^2_2+|w(T)|^2),
\ee
where the last equality follows from the estimates for $\delta y$ and $\eta$ obtained in Lemma~\ref{suf-cond.p-7-lem}.\\
{\em Contribution of $\varphi_1.$}
Integrating by parts in time,
we can write the contribution of $\varphi_1$ as 
\be
\label{varphi1cont}
\half \int_Q \ddt (w\cdot b)^2 p \delta y
=
\half \int_\Om (w(T)\cdot b)^2 p(x,T) \delta y(x,T) -
\half \int_Q (w\cdot b)^2 \ddt  (p \delta y)
\ee
The contribution of the term at $t=T$
is of the desired order.
Let us proceed with the estimate for the last term in the r.h.s. of \eqref{varphi1cont}. We have
\begin{equation}
\begin{aligned}
\label{dotpdeltay}
 \ddt (p\delta y) &= (-\delta y\Delta p+p\Delta\delta y) \\
 &+\left(- (\yb-y_d)-\sum_{j=1}^q c_j\dot{\mu}_j\right) \delta y  
 +  \left( \sum_{i=1}^m v_i b_i y -3\gamma
\yb(\delta y)^2 -\gamma (\delta y)^3\right) p.
\end{aligned}
\end{equation}
For the contribution of first term in the r.h.s. of latter equation we get
\be
\int_Q (w\cdot b)^2(-\delta y\Delta p+p\Delta\delta y) = \sum_{i,j=1}^m \int_0^T w_iw_j \int_\Omega \nabla (b_ib_j) \cdot(\delta y \nabla p -p\nabla\delta y).
\ee
Using \cite[Lem. 2.2]{ABK-PartI}, 
since $\nabla (b_ib_j)$ is essentially bounded for every pair $i,j$,
it is enough to prove that 
\be
\int_\Omega \nabla (b_ib_j) \cdot(\delta y \nabla p -p\nabla\delta y) \to 0
\ee
uniformly in time.
For this, in view of the estimate for $\|\delta y\|_{L^\infty(0,T;H^1_0(\Om))}$ obtained in Lemma \ref{suf-cond.p-7-lem} item {\rm (ii.b)}, and since $p$ is essentially bounded, it suffices to prove that $p$ is in $L^\infty(0,T;H^1(\Omega))$ which follows from Corollary \ref{cor:reg-p}.

Let us continue with the expression in \eqref{dotpdeltay}. The terms containing $\delta y$ go to 0 in $L^\infty(0,T;L^2(\Omega))$ and that is sufficient for our purpose. The only term that has to be estimated is
\begin{multline}
\label{wb3}
\int_Q (w\cdot b)^2 (v\cdot b) yp = \frac13 \int_Q \frac{\dd}{\dd t} (w\cdot b)^3 yp \\= \frac13\int_\Omega (w(T)\cdot b)^3y(\cdot,T)p(\cdot,T)- \frac13\int_Q (w\cdot b)^3 \frac{\dd}{\dd t}(yp).
\end{multline}
We consider the pair of state and costate equations with $g:=y-y_d$ given as
\begin{equation}
 \begin{aligned}
\dot y - \Delta y +\gamma y^3  &=  (u\cdot b)y  + f;& y(0)&=y_0;\\
- \dot p -\Delta  p + \gamma y^2 p&=  (u\cdot b)p  + g+ c \dot \mu;& p(T)&=0.
 \end{aligned}
 \end{equation}
and so for sufficiently smooth $\varphi\colon \Om  \times (0,T)\rar \RR$ we have
 \begin{equation}
 \begin{aligned}
  \int_Q\varphi &  \frac{\dd }{\dd t} (yp)   = \int_Q \varphi (\dot y p + y \dot p)   \\
  &= \int_Q \varphi \left[ (\Delta y -\gamma y^3 +  (u\cdot b)y  + f)p + y  (-\Delta p + \gamma y^2 p  -  (u\cdot b)p  - g- c \dot \mu) \right]  \\
  & =\int_Q \varphi  \left[ f p -  y  g +c \dot \mu { y} \right]+
  {\nabla \varphi \cdot (-p\nabla y + y\nabla p ), }
 \end{aligned}
 \end{equation}
 and, consequently, 
 we have for $\varphi=(w \cdot b)^3,$ 
 \begin{equation}
 \begin{aligned}
  \int_Q (w \cdot b)^3 \frac{\dd }{\dd t} (yp) 
  & =\int_Q  (w \cdot b)^3 \left[ f p - y  g +c \dot \mu { y} \right] 
  + {\nabla  (w \cdot b)^3 \cdot (-p\nabla y + y\nabla p )}. 
 \end{aligned}
 \end{equation}
 By Hypothesis \ref{hyp-setting}, $f$ and $b$ are sufficiently smooth, $\dot \mu$ is essentially bounded,
 $y, p \in L^{\infty}(0,T;H^1_0(\Om))$. 
 We estimate
 \begin{equation*}
\begin{aligned}
&\left| \int_Q (w\cdot b)^3 \frac{\dd}{\dd t}(yp) \right|  \leq
\|b\|^3_{{\infty}} \|w\|_\infty \|w\|_2^2 
\left\|f p -  y  g +c \dot \mu {y}\right\|_{L^\infty(0,T;L^1(\Om))}\\
&+  O(\|b\|_\infty^2\|\nabla b\|_\infty)  \|w\|_\infty \|w\|_2^2 
\left(\left\|y\right\|_{L^\infty(0,T;H^1_0(\Om))} \left\|p\right\|_{L^\infty(0,T;H^1_0(\Om))}\right)=o(\norm{w}{}^2).
\end{aligned}
\end{equation*}
\if{So we need
\be
\left\|\dot{y}\right\|_{L^\infty(0,T;L^2(\Omega))} \leq M,\quad
\left\|\dot{p}\right\|_{L^\infty(0,T;L^2(\Omega))}\leq M.
\ee
}\fi
\\
{\em Contribution of $\varphi_0.$}
  Integrating by parts, we have that
\be
\label{varphi0}
\begin{split}
\int_0^T w_i \int_\Om b_i  \varphi_0 &=
\int_0^T w_i\int_\Om b_i  (p \Delta \eta - \eta \Delta p) 
=
\int_0^T w_i\int_\Om \nabla b_i \cdot (-p \nabla \eta + \eta \nabla p)\\
&{ =\int_0^T w_i\int_\Om \Big(p\eta\Delta b_i +2 \eta \nabla p \cdot \nabla b_i).}
\end{split}
\ee
Recalling that $b\in W^{2,\infty}(\Om)$ (see \eqref{sonc:setting-t1-1}) and that $p$ is essentially bounded (due to Theorem~\ref{sonc:setting-t1}), we get for the first term  in the r.h.s. of the latter display, 
\be
\left| \int_0^T w_i\int_\Om p\eta\Delta b_i \right| \leq \|\Delta b_i\|_\infty\|w_i\|_2 \|p\|_\infty \|\eta\|_{L^2(0,T;L^2(\Om))},
\ee
that is a small-$o$ of $\|w\|_2^2$ in view of item {\rm (iii.a)} of Lemma \ref{suf-cond.p-7-lem}.
For the second term in the r.h.s. of \eqref{varphi0} we get
\be
\left| \int_0^T w_i\int_\Om \eta \nabla p \cdot \nabla b_i \right| \leq
\|\nabla b_i\|_{\infty} \|w_i\|_2 \|\eta\|_{L^2(0,T;L^2(\Om))} \|\nabla p\|_{L^\infty(0,T;L^2(\Om)^n)}
\ee
Since $p\in L^\infty(0,T;H^1(\Omega))$ as showed some lines above and in view of item {\rm (iii.a)} of Lemma \ref{suf-cond.p-7-lem}, we get that the r.h.s. of latter equation is a small-$o$ of $\|w\|_2^2,$ as desired.

Collecting the previous estimates, we get 
\eqref{suf-cond.p-2-1}.
Similarly, since 
$\delta y \rar 0$ uniformly and
 the costate $p$ is
essentially bounded, with 
\eqref{suf-cond.p-7-lem-eq}(i) we get
\be
\label{suf-cond.p-8}
\left| \int_Q  pb (\delta y)^3 \dd x\dd t \right| 
= 
o \left( \|\delta y\|^2_2 \right) =
o \left(  \|w\|_2 ^2+ |w(T)|^2 \right).
\ee
The result follows.
\end{proof}

\begin{corollary}\label{expansion:J}
Let $u=\ub+v$ be an admissible control. 
Then, setting
$w(t):=\int_0^t v(s)\dd s$,
we have the reduced cost expansion
\be
F(u)=F(\ub) +  D F(\ub)v +
O(\|w\|^2_2 + |w(T)|^2).
 \ee
\end{corollary}

\begin{proposition}\label{prop-1}
 Let \blue{$(p,\mu) \in \Lambda_1,$} and let $(z,v) \in Y \times L^2(0,T)^m$ satisfy the linearized state equation \eqref{lineq2}. Then,
\be
\begin{split}
  \int_0^T \Psi^p(t) \cdot v(t) \dd t  =   DJ(\ub,\yb)(z,v)+
 \sum_{j=1}^q \int_0^T g_j'(\yb(\cdot,t)z(\cdot,t)  {\rm d} \mu_j(t),
\end{split} 
\ee
where 
$$
DJ(\ub,\yb)(z,v)=\sum_{i=1}^m\int_0^T  \alpha_i v_i \dd t+ \int_Q(\yb-y_d)z\dd x\dd t + \int_\Omega (\yb(T)-y_{dT}) z(T)\dd x,$$ 
and it coincides with $DF(\ub)v.$
\end{proposition}
\begin{proof}
It follows from \eqref{lineq2}, \eqref{costat-eq} and \eqref{def-psi}.
 \if{From   we obtain
 \begin{multline*}
\int_Q p (v\cdot b \yb) {\rm d}x {\rm d}t  \\= 
 \sum_{j=1}^q \int_0^T g_j'(\yb)z  {\rm d} \mu_j(t)  
  + \int_Q  (\yb-y_d) z {\rm d}x {\rm d}t
 +  \int_\Omega (\yb(x,T)-y_{dT}(x)) z(x,T)  {\rm d}x.
\end{multline*} 
We conclude with \eqref{def-psi}.}\fi
\end{proof}

\subsection{Proof of Theorem \ref{ThmSC}}
What remains to prove is similar to what has been done in Aronna, Bonnans and Goh \cite[Theorem 5]{MR3555384}, 
in a finite dimensional setting,
except that here the control variable may be multidimensional and in \cite{MR3555384} it is scalar.

We start by showing item a).
If the conclusion does not hold,
there must exist a sequence $(u_{\ell},y_\ell)$
of admissible trajectories, with $u_\ell$ distinct from $\ub,$
such that $v_{\ell}:=u_{\ell}-\ub$
converges to zero in $L^2(0,T)^m$, 
and  
\be
\label{inequ1}
J(u_{\ell},y_\ell) \le J(\ub,\yb) + o(\Upsilon_{\ell}^2),
\ee
where $(w_\ell,\zeta_\ell)$ is obtained by Goh transform \eqref{Goh}, $h_{\ell}:=w_{\ell}(T)$ and  
$$\Upsilon_{\ell} := \sqrt{ \|w_{\ell}\|_2^2 + |w_{\ell}(T)|^2}.
$$
Let \blue{$(p,\mu)\in \Lambda_1$}.
Adding 
$\int_0^T g(y_{\ell}) \dd \mu\le 0$ on both sides of \eqref{inequ1} leads to
\be
\label{equ-exp:L}
\call[p,\mu](u_\ell , y_\ell) \leq \call[p,\mu](\ub,\yb) + o(\Upsilon_{\ell}^2).
\ee
Set $(\vb_{\ell}, \wb_{\ell},\hb_{\ell}) := 
(v_{\ell},w_{\ell},h_{\ell})/\Upsilon_{\ell}$.
Then $(\wb_{\ell},\hb_{\ell})$ has unit norm in $L^2(0,T)^m \times \RR^m$. Extracting if necessary a subsequence,
we may assume that there exists $(\wb, \hb)$
in
$L^2(0,T)^m \times \RR^m$ such that
\be
\wb_{\ell}\rightharpoonup  \wb \quad \text{and} \quad  h_{\ell}\rightarrow  \hb,
\ee
where the first limit is given in the weak topology of $L^2(0,T)^m$.
Set $\bar\zeta:=\zeta[\wb].$ The remainder of the proof is split in two parts:

\textbf{Fact 1:} The triple $(\zetab,\wb,\hb)$ belongs to $PC_2^*$ (defined in \eqref{p-cal}).

\textbf{Fact 2:} The inequality \eqref{inequ1} contradicts the hypothesis of uniform positi\-vi\-ty~\eqref{Q-coerc_cond}.

\textbf{Proof of Fact 1.}  We divide this part in four steps: {\bf (a)} $\wb_i$ is constant on each maximal arc of $I_i$, for $i=1,\dots,m,$ {\bf (b)} \eqref{lem-pc2-1}(i),(ii) hold, {\bf (c)} \eqref{w-C} holds, and {\bf(d)} \eqref{jumpcond}  holds.

\textbf{(a)} From Proposition \ref{suf-cond.p}, inequality \eqref{equ-exp:L}, and \eqref{FirstControl} we have
\be
\label{equ2}
- \whq[p,\dd\mu](\zeta_{\ell},w_{\ell},h_{\ell})+ o(\Upsilon^2_{\ell}) \ge \sum_{i=1}^m \int_0^T \Psi^p_i(t)\cdot v_{\ell,i}(t) \dd t \ge 0.
\ee
By the continuity of the quadratic form $\whq[p,\dd\mu]$ 
over the space $L^2(0,T)^m \times \RR^m,$ 
we deduce that
\begin{align}
 0 \le \int_0^T \Psi^p_i(t) v_{\ell,i}(t) \dd t \le O(\Upsilon^2_{\ell}),\quad \text{for all } i=1,\dots, m.
\end{align}
Hence, since the integrand in previous inequality is nonnegative for all $\ell \in \NN$, we
have that
\begin{align}\label{equ1}
 \lim_{\ell \rightarrow \infty }\int_0^T 
\Psi^p_i(t)\varphi(t) \vb_{\ell,i}(t) \dd t = 0
\end{align}
for any nonnegative  $C^1$ function $\varphi \colon [0,T] \rightarrow \RR$. 
Let us consider, in particular, $\varphi$ having its support $[c,d]\subset I_i$. Integrating by parts in \eqref{equ1} and using that $\wb_{\ell}$ is a
primitive of $\vb_{\ell}$, we obtain
\begin{align}
 0= \lim_{\ell\rar \infty}  \int_0^T \frac{\dd}{\dd t}(\Psi^p_i \varphi )\wb_{\ell,i}\dd t=\int_c^d \frac{\dd}{\dd t}  (\Psi^p_i(t)\varphi )\wb_i \dd t.
\end{align}
Over $[c,d]$, $\vb_{\ell,i}$ has constant sign and, therefore, 
$\wb_i$ is either nondecreasing or 
nonincreasing. Thus, we can integrate by parts in the latter equation to get
\be
\int_c^d \Psi^p_i(t)\varphi(t)  \dd \wb_i(t)  = 0.
\ee
Take now any $t_0 \in (c,d)$. {Assume, w.l.g. that $t_0\in \check{I}_i.$} By the strict complementary condition for the 
control constraint given in \eqref{strict-compl}, there exists a multiplier such that the associated $\Psi^p$ verifies $\Psi^p_i(t_0)> 0$. Hence, in view of the continuity of $\Psi^p_i$ on $I_i$, there exists 
$\varepsilon >0$ such that $\Psi^p_i>0$ on $(t_0 - 2\varepsilon,t_0 + 2\varepsilon) \subset (c,d)$.
Choose $\varphi$ such that $\supp \varphi \subset (t_0 - 2\varepsilon,t_0 + 2\varepsilon)$, 
and $\Psi^p_i\varphi \equiv1$ on $(t_0 - \varepsilon, t_0 + \varepsilon)$, {then $\wb_i(t_0+\eps)-\xb_i(t_0-\eps) =0.$ Since $\dd \wb_i \ge 0$,
we obtain
$\dd \wb_i =0$ a.e. on $\check I_i.$
Since $t_0$ is an arbitrary point in the interior of $I_i,$ we get
\be
\label{dw-zero}
\dd \wb_i =0\quad \text{ a.e. on }\check I_i.
\ee
This concludes step {\bf (a)}.}

\noindent\textbf{(b)} 
We now have to prove \eqref{lem-pc2-1}(i),(ii). Assume now that $B_0 \neq \emptyset$ or, w.l.g., that $\check{B}_0\neq \emptyset,$ and let $i \in \check{B}_0$.
By the previous step, 
$\wb_i$ is equal to some
constant $\theta$ a.e. over $(0,\tau_1)$.
Let us show that 
$\theta = 0$.
By the strict complementarity 
condition for 
the control constraint \eqref{strict-compl} there exist $t, \delta > 0$ {and a  multiplier such that the associated $\Psi^p$} satisfies $\Psi_i^p > \delta$ 
on $[0, t]\subset [0,\tau_1)$. 
{
By considering in \eqref{equ1} a 
nonnegative Lipschitz continuous function $\varphi \colon [0,T] \rar \RR$ being equal to  $1/\delta$ on $[0,t]$, with support included in $[0,\tau_1),$ and since $\vb_{\ell,i} \ge 0$ a.e. on $[0,\tau_1],$ we obtain, for any $\tau\in [0,t],$
\be
\wb_{\ell,i}(\tau) =\int_0^\tau  \vb_{\ell,i}(s) \dd s \leq \int_0^t \Psi_i^p(s) \varphi(s) \vb_{\ell,i}(s) \dd s    \rar 0,\quad \text{when }\ell \to \infty.
\ee
Thus $\wb_i=0$ a.e. on $[0,t].$ Consequently, from \eqref{dw-zero} we get $\wb_i=0$ a.e. on $[0,\tau_1).$}
The case when 
$i \in  B_{r-1} $ follows by a similar argument. This yields item {\bf(b)}.

\noindent{\bf(c)} {Let us prove \eqref{w-C}. 
We have, since $y_\ell$ is admissible and  $g$ linear, 
\be
0 \ge g_j(y_{\ell}(\cdot,t)) - g_j(\yb(\cdot,t)) = \int_\Omega c_j(x)(y_\ell-\yb)(x,t)\dd x,\quad \text{on } [\tau_k,\tau_{k+1}],
\ee
whenever $k,j$ are such that $k\in \{0,\dots,r-1\}$ and $j\in C_k.$
Let $z_\ell$ denote the linearized state corresponding to $v_\ell$, and let $\eta_\ell := y_\ell-\yb-z_\ell.$ By Lemma \eqref{suf-cond.p-7-lem}(iii), we deduce that
\be
\label{equ3}
\int_\Omega c_j(x)  z_\ell(x,t)\dd x \leq - \int c_j(x)\eta_\ell(x,t)\dd x \leq o(\Upsilon_\ell),\quad \text{on } [\tau_k,\tau_{k+1}].
\ee
Thus, by the Goh transform \eqref{Goh},
\be
\label{equcj}
\int_\Omega c_j(x) (\bar{\zeta}_\ell(x,t)+B(x,t)\cdot \wb_\ell(t))\dd x \leq o(1),\quad \text{on } [\tau_k,\tau_{k+1}],
\ee
where $\bar\zeta_\ell$ is the solution of \eqref{zeta} corresponding to $\wb_\ell.$
Let $\varphi$ be some time-dependent nonnegative continuous function with 
support included in $I_j^C.$
From \eqref{equcj}, we get that
\be
\label{equcj2}
\int_{\tau_k}^{\tau_{k+1}} \varphi \int_\Omega c_j (\bar{\zeta}_\ell+B\cdot \wb_\ell)\dd x \dd t\leq o(1).
\ee
Taking the limit $\ell \to \infty$ yields
\be
\label{equcj3}
\int_{\tau_k}^{\tau_{k+1}} \varphi \int_\Omega c_j (\bar{\zeta}+B\cdot \wb)\dd x \dd t\leq 0,
\ee
where $\bar\zeta$ is the solution of \eqref{zeta} associated to $\wb.$ Since \eqref{equcj3} holds for any nonnegative $\varphi,$ we get that 
\be
\label{linstateconsleq0}
\int_\Omega c_j (\bar{\zeta}(x,t)+B(x,t)\cdot \wb(t))\dd x \leq 0,\quad \text{a.e. on } [\tau_k,\tau_{k+1}].
\ee
In particular, if $T\in I_j^C,$ we get from \eqref{equcj} that
\be
\label{equcj4}
\int_\Omega c_j (\bar{\zeta}(x,T)+B(x,T)\cdot \hb)\dd x \leq 0.
\ee
We now have to prove the converse inequalities in \eqref{equcj3} and \eqref{equcj4}. 
}

By Proposition \ref{prop-1} and since $u_\ell$ is admissible, we have
\be\label{equ4}
\begin{split}
 \sum_{j=1}^q \int_0^T g_j'(\yb(\cdot,t))z(\cdot,t)  {\rm d} \mu_j(t) +  DJ(\ub,\yb)(z,v) =  \int_0^T \Psi^p(t) \cdot v_\ell(t) \dd t\ge 0.
\end{split} 
\ee  
By Proposition \ref{expansion:J}, we have 
$
F(u_\ell)=F(\ub) + DF(\ub)v_\ell + o(\Upsilon_{\ell}).
$
This, together with \eqref{equ4}, yield
\be
0 \le F(u_\ell) - F(\ub) + o(\Upsilon_{\ell}) +  \sum_{j=1}^q \int_0^T g_j'(\yb(\cdot,t))z(\cdot,t) {\rm d} \mu_j(t).
\ee
Using \eqref{inequ1} in latter inequality implies that 
\be
 -o(\Upsilon_\ell) \leq  \sum_{j=1}^q \int_0^T g_j'(\yb(\cdot,t))z(\cdot,t)  {\rm d} \mu_j(t),
\ee
thus
\be
 o(1) \leq  \sum_{j=1}^q \int_0^T g_j'(\yb(\cdot,t))(\bar\zeta_\ell(\cdot,t) + B(\cdot,t)\cdot \wb_\ell(t))  {\rm d} \mu_j(t).
\ee
Since, for every $j=1,\dots,q,$ the measure $\dd \mu_j$ has an essentially bounded density over $[0, T )$ (in view of Theorem \ref{sonc:setting-t1}), we have that
\be\label{equ5}
\begin{aligned}
 &0  \le \liminf_{\ell \rar \infty}  \sum_{j=1}^q  \int_{[0,T]} g_j'(\yb(\cdot,t))(\zetab_\ell(\cdot,t) + B(\cdot,t)\cdot\wb_{\ell}(t))  {\rm d} \mu_j\\
& = \lim_{\ell\rar \infty} \sum_{j=1}^q  \int_{[0,T)} g_j'(\yb(\cdot,t))( \zetab_{\ell}(\cdot,t)+ B(\cdot,t) \cdot\wb_\ell(t)){\rm d} \mu_j.\\
\end{aligned}
\ee
Using \eqref{linstateconsleq0} and the strict complementarity for the state constraint \eqref{strict-compl}(ii), we get \eqref{w-C}. This concludes the proof of item {\bf (c)}.
\\
\noindent{\bf (d)} Let us now prove \eqref{jumpcond}. Assume that $j\in\{1,\dots,q\}$ is such that $T \in I^C_j.$  One inequality was already proved in \eqref{equcj4}. If we further have that $[\mu_j(T )] > 0,$  condition \eqref{jumpcond} follows from \eqref{equ5}.

We conclude that the limit direction $(\zetab,\wb,\hb)$ belongs to $PC_2^*.$

\textbf{Proof of Fact 2}. From Proposition \ref{suf-cond.p} we obtain 
\begin{equation}
\begin{aligned}
\label{equ6}
\whq[p,\dd\mu] &(\zeta_\ell,w_{\ell}, h_{\ell}) \\
&= \call[p,\mu](u_{\ell},y_\ell) - \call[p,\mu](\ub,\yb) -
\int_0^T \Psi^p \cdot v_{\ell}\dd t + o(\Upsilon^2_{\ell} )\le o(\Upsilon^2_{\ell} ),
\end{aligned}
\end{equation}
where the last inequality follows from \eqref{equ-exp:L} and since $\Psi^p \cdot v_{\ell}\ge 0$ a.e. on~$[0, T ]$ in view of  the first order condition \eqref{FirstControl}.
Hence,
\be
\liminf_{\ell\rar \infty} \whq[p,\mu] (\bar\zeta_\ell,\yb_{\ell}, \hb_{\ell}) \le \limsup_{\ell \rar \infty}  \whq[p,\mu] (\zetab_\ell,\wb_{\ell}, \hb_{\ell}) \le 0.
\ee
Let us recall that, in view of the hypothesis (iii) of the current theorem, the mapping $\whq[p,\dd\mu]$ is a Legendre form in the
Hilbert space $\{(\zeta[w],w, h) \in Y\times L^2(0,T)^m \times \RR^m\}$. Furthermore, for
the critical direction $(\zetab, \wb, \hb)$, due to the uniform positivity condition \eqref{Q-coerc_cond}, there is a multiplier $(\bar p,\bar{\mu}) \in \Lambda_1$ such that
\be\label{equ7}
\rho (\|\wb\|^2_2+|\hb|^2)
\le \whq[\bar p,\bar{\mu}](\zetab, \wb, \hb) = \liminf_{\ell \rar \infty} \whq[\bar p,\bar{\mu}] ( \zetab_\ell,\wb_{\ell}, \hb_{\ell}) \le 0,
\ee
where the equality holds since $\whq[\bar p,\bar{\mu}]$ is a Legendre form and the  inequality is due to \eqref{equ6}. From \eqref{equ7} we get  $( \wb, \hb) = 0$ and $\ds\lim_{k\rar \infty} \whq[\bar p,\bar{\mu}]( \zetab_\ell,\wb_{\ell}, \hb_{\ell}) =0$.
Consequently, $(\wb_{\ell}, \hb_{\ell})$ converges strongly to $( \wb, \hb) = 0$ which is a contradiction,
since $( \wb_{\ell}, \hb_{\ell})$ has unit norm in $L^2(0,T)^m \times \RR^m$. We conclude that $(\ub, \yb)$ is
an $L^2$-local solution 
satisfying the weak quadratic growth condition.

Conversely, assume that the weak quadratic growth condition \eqref{growth-cond} holds at $(\ub,\yb)$ for $\rho>0.$ Note that $(\ub,\yb,\wb),$ with $\wb(t)=\int_0^t \ub(s) \dd s,$ is a 
$L^2$-local solution of the problem
\be
\label{Paux}
\begin{split}
\min_{u\in \calu_{\rm ad}}\, &J(u,y[u]) - \rho \left( \int_0^T (w - \wb)^2 \dd t + |w(T)-\wb(T)|^2\right),\\
{\rm s.t.\,}&\,\,\dot w = u,\,\, w(0)=0,\,\text{\eqref{stateconstraint} holds},\\
\end{split}
\ee
Applying the second order necessary condition in Theorem \ref{SONC} 
to this problem \eqref{Paux}, followed by the  Goh transform, yields the uniform positivity \eqref{Q-coerc_cond2}. For further details we refer to the corresponding statement for ordinary differential equations in  \cite[Theorem 5.5]{ABDL12}.
$\square$

\if{
\appendix
\section{Proof of Proposition \ref{cc-m1-p}}
\begin{proof}
(i)
Over each arc we have at most one active state constraint,
in view of the uniform local controllability condition
\eqref{controllability}.
On singular arcs 
(i.e., when no constraint is active) the operators
$\cala$, $\calb$ reduce to zero so that no continuity relations of
type \eqref{LemmaJ-2} can be obtained over junctions with a
singular arc. At other junctions we obtain one relation which means that 
$w$ is continuous. 
The last two conditions are easily obtained.
So, the l.h.s. of 
\eqref{cc-m1-p1}
is included in the r.h.s. \\
(ii)  
Given $(w_0,h_0)$ in the r.h.s. of \eqref{cc-m1-p1},
we need to construct 
$(w,h) \in PC$, arbitrarily close to 
$(w_0,h_0)$ in $L^2(0,T)^m\times\RR$.
As a first step, given $\eps>0$, 
let $w_\eps$ be equal to  $w_0$ on B and C arcs, $C^1$ on S arcs, continuous  on $[0,T],$ and such that $\|w_\eps-w_0\|_2 \leq \eps$,
and $w_\eps(T)=h_0$ if the last arc,  i.e. $(\tau_{r-1},\tau_r),$ is singular.
This is possible since it requires modifying $w_0$ only on S arcs,
 maintaining the values at junction points. \\
Take 
$h_\eps:=w_\eps(T)$.
{Next, compute the pair 
$(\zeta_\eps,w'_\eps)\in Y\times L^2(0,T)^m$
satisfying \eqref{zeta},
$w'_\eps(t)= w_\eps(t)$ over B and S arcs, 
and over $C$ arcs, 
$w'_\eps(t)$ is such that the linearized 
state constraint is active, i.e.,
$\int_\om c(x) (\zeta_\eps(x,t) ++B(x,t)\cdot w'_\eps(t)=0$.
}
Then $w'_\eps$
belongs to the set $PC''(\ub)$, defined as $PC$, but relaxing the
continuity condition of $w$ at junctions, and endowed with the 
piecewise (on each arc) $H^1$-norm, denoted as $\|w\|_{H^1_P}$. 
By induction over the arcs and using the 
uniform local controllability condition
\eqref{controllability}, 
setting $h'_\eps := w'_\eps(T)$, 
we easily obtain that 
\be
\label{wp-eps-w-eps}
\| w'_\eps - w_\eps\|_{H^1_P} 
+ | h'_\eps-h_\eps| = O(\eps).
\ee
For $w\in PC''(\ub)$,
let $Aw := \{ w(\tau_+)-w(\tau_-); \; \tau\in\calt\}$
where $\calt$ is the set of junction points. This operator is
 well defined, and $PC = PC''(\ub) \cap \Ker A$. 
By \eqref{wp-eps-w-eps},  $| A w'_\eps | = O(\eps)$.
So, by the infinite dimensional 
\cite[Thm. 2.200]{MR1756264} 
version of Hoffman's Lemma \cite{MR0051275}, 
there exists $w''_\eps\in PC$ such that
\be
\|w''_\eps - w'_\eps \|_{H^1_P} = O( | A w'_\eps |) = O(\eps).
\ee
Consequently,
$
\|w''_\eps - w_\eps \|_{2}
\leq 
\|w''_\eps - w'_\eps \|_{2} + \|w'_\eps - w_\eps \|_{2} 
= O(\eps)$,
and setting
$h''_\eps := w''_\eps(T)$, 
we can estimate 
$|h''_\eps - h_\eps \|_{2}$
in the same way.
The conclusion follows.
\end{proof}
}\fi

\appendix

\section{Well-posedness of state equation and existence of optimal controls}\label{sec:well-posed}
In this section we recall some statements from \cite{ABK-PartI}, for proofs we refer to the latter reference.
\if{We start by the following  technical result which follows immediately by  H\"older's inequality.

\begin{lemma}
\label{lem-uby}
The state equation \eqref{dynamics}
has a unique solution 
$y=y[u,y_0,f]$ in $Y$. The mapping $(u,y_0,f) \mapsto  y[u,y_0,f]$ is 
$C^\infty$ from $L^2(0,T)^m  \times H^1_0(\Om) \times L^2(Q)$ to $Y$,
and nondecreasing w.r.t. $y_0$ and $f$.
In addition, there exist functions $C_i$, $i=1$ to 2,
not decreasing w.r.t. each component, such that 
\begin{gather}
\label{state-equ-estimA.l-0}
\| y \|_{ L^\infty (0,T;L^2(\Om) ) } +
\| \nabla y \|_2 \leq 
C_1( \|y_0\|_2, \| f \|_2, \|u\|_2 \| b \|_\infty),
\\ 
\label{state-equ-estimA.l-1}
\|y\|_Y \leq 
C_2( \|y_0\|_{ H^1_0(\Om) }, \| f \|_2, \|u\|_2\| b \|_\infty).
\end{gather} %
Moreover, the state $y$ also belongs to $C([0,T];H^1_0(\Omega))$, since $Y$ is continuously embedded in that space \cite[Theorem 3.1, p.23]{LioMag68a}.
\end{lemma}

}\fi
\if{
\begin{proof}
Indeed, by H\"older's inequality
{XXX do you agree for deleting this sentence ? }
 (using 1/2 = 1/3 + 1/6), for each $i=1,\dots,m,$
\be
\int_Q u^2_i b_i^2y^2 \dd x \dd t
= \disp
\int_0^T u^2_i(t) 
\| b_iy(\cdot,t) \|^2_{L^2(\Om)} \dd t
\leq \disp
\left( \int_0^T u^2_i \dd t\right) 
\| b_i y\|^2_{L^\infty(0,T;L^2(\Om))}.
\ee
This clearly implies \eqref{lem-uby1}.
The conclusion easily follows.
\end{proof}
}\fi

\begin{lemma}
\label{state-equ-estimA.l}
The state equation \eqref{dynamics}
has a unique solution 
$y=y[u,y_0,f]$ in $Y$. The mapping \blue{$(u,y_0,f) \mapsto  y[u,y_0,f]$} is 
$C^\infty$ from $L^2(0,T)^m  \times H^1_0(\Om) \times L^2(Q)$ to $Y$,
and nondecreasing w.r.t. $y_0$ and $f$.
In addition, there exist functions $C_i$, $i=1$ to 2,
not decreasing w.r.t. each component, such that 
\begin{gather}
\label{state-equ-estimA.l-0}
\| y \|_{ L^\infty (0,T;L^2(\Om) ) } +
\| \nabla y \|_2 \leq 
C_1( \|y_0\|_2, \| f \|_2, \|u\|_2 \| b \|_\infty),
\\ 
\label{state-equ-estimA.l-1}
\|y\|_Y \leq 
C_2( \|y_0\|_{ H^1_0(\Om) }, \| f \|_2, \|u\|_2\| b \|_\infty).
\end{gather} %
Moreover, the state $y$ also belongs to $C([0,T];H^1_0(\Omega))$, since $Y$ is continuously embedded in that space \cite[Theorem 3.1, p.23]{LioMag68a}.
\end{lemma}

\if{
\begin{lemma}
\label{state-equ-estimA.l}
Given $\fh \in L^2(Q),$ the equation 
\be
\label{dynamicsA}
\left\{
\begin{aligned}
&\dot y(x,t) - \Delta y(x,t) +\gamma y^3(x,t) = \fh(x,t)\quad \text{in } Q,\\
&y=0\,\, \text{on } \Sigma,\quad 
y(\cdot,0) = y_0\;\;\text{in } \Om,
\end{aligned}
\right.
\ee
has a unique solution in $Y$, and we have that 
\begin{gather}
\label{state-equ-estimA.l-0}
\| y \|_{L^\infty(0,T;L^2(\Om))} +
\| \nabla y \|_{2} \leq c
\left( \|y_0\|_{2} + \| \fh \|_{2} \right),
\\
\label{state-equ-estimA.l-0-2}
{\|y \|_{H^1(0,T;L^2(\Om))} +
\| \nabla y \|_{L^\infty(0,T;L^2(\Om))} \leq c
\left( \| y_0\|_{H^1_0(\Om)} + \| \fh \|_{2} \right),}
\\
\label{state-equ-estimA.l-1}
\|y\|^2_Y \leq c \left(
\|\fh \|^2_{2} +
\|y_0\|^2_{H^1_0(\Om)}
+ \gamma \|y_0\|^4_{H^1_0(\Om)} 
\right).
\end{gather}
\end{lemma}

\if{
\begin{lemma}
\label{state-equ-estim.l}
Given $u\in \Uspace,$ the state equation \eqref{dynamics}
has a unique solution $y[u]$ that satisfies
\be
\label{state-equ-estim.l-1}
\| y[u] \|^2_{L^\infty(0,T;L^2(\Om))} 
\leq {M_0};
\ee
\be
\label{state-equ-estim.l-2}
\| y[u] \|^2_Y \leq c \left( \| f  \|^2_2  
+{M_0\sum_{i=0}^m \|u_i \|^2_2 \|b_i \|^2_{\infty}} 
+ \| y_0\|^2_{H^1_0(\Om)} + \gamma \| y_0\|^4_{H^1_0(\Om)}
 \right),
\ee
where ${M_0 := e^T e^{2 \sum_{i=0}^m \| u_i \|_{1} \|b_i\|_{\infty}}\left( \|y_0\|^2_{2} + \|f\|^2_2 \right)}$.
\if{\be
\label{state-equ-estim.l-1}
\| y[u] \|^2_{L^\infty(0,T;L^2(\Om))} 
\leq M_0;
\ee
\be
\label{state-equ-estim.l-2}
\| y[u] \|^2_Y \leq c \left( \| f  \|^2_2  
+M_1 
+ \| y_0\|^2_{H^1_0(\Om)} + \gamma \| y_0\|^4_{H^1_0(\Om)}
 \right),
\ee
where $
M_0 := 
e^T e^{2 \sum_{i=0}^m \| u_i \|_{1} \|b_i\|_{\infty}}\left( \|y_0\|^2_{2} + \|f\|^2_2 \right)
,$
$M_1 := 
\sum_{i=0}^m \|u_i \|^2_2 \|b_i \|^2_{\infty} M_0.
$}\fi
\end{lemma}

\begin{proof}
Estimate \eqref{state-equ-estim.l-1} follows, as in Lemma \ref{state-equ-estimA.l}, by multiplying the state equation by $y=y[u]$, then integrating in space and finally applying Gronwall's Lemma.
\if{ Let us split the proof in two parts. First part: proof of \eqref{state-equ-estim.l-1}.
Multiplying the state equation by $y=y[u]$ and integrating in space we
obtain that for a.a. $t\in (0,T)$
\begin{multline}
\label{se-est}
\half \frac{\dd}{\dd t} \norm{y(\cdot,t)}{L^2(\Om)}^2 + \norm{\nabla y(\cdot,t)}{L^2(\Om)}^2 +
 \gamma \int_\Om y(x,t)^4  \\
 = \int_\Om f(x,t)y(x,t)  + 
\sum_{i=0}^m 
u_i(t) \int_\Om b_i(x) y^2(x,t).
\end{multline}
So, $\nu(t) := \norm{y(\cdot,t)}{L^2(\Om)}^2$ satisfies
\be
\half \dot \nu(t) \leq \half \|f(\cdot,t) \|^2_{L^2(\Om)} 
+
\left( \frac12  + \sum_{i=0}^m |u_i(t)|  \|b_i\|_\infty \right)
\nu(t).
\ee
Integrating over $(0,T)$, we get 
\be
\nu(t) \leq \|y_0\|^2_{2}+ \|f\|^2_2 
+
\int_0^t \left(1 + 2 \sum_{i=0}^m 
|u_i(s)| \|b_i\|_\infty\right) \nu(s) \dd s.
\ee
We conclude with the integral form of Gronwall's lemma.}\fi

Let us now prove \eqref{state-equ-estim.l-2}.
We apply Lemma \ref{state-equ-estimA.l}
with $\fh:= f+ {\sum_{i=0}^m u_i b_i } y$, 
so that, by Lemma \ref{lem-uby}:
\be
\|\fh\|^2_2 \leq 2 \|f\|^2_2
 + {2} 
\sum_{i=0}^m \|u_i \|^2_2 
\|b_i\|^2_\infty 
\|y\|^2_{L^\infty(0,T;L^2(\Om))}.
\ee
The result follows, using the definition of {$M_0$ given in the statement.} 

It remains to prove the uniqueness of the solution. So, let 
$y'$ and $y''$ in $Y$, be two states associated with
$u\in L^2(0,T)$. Set $z:=y''-y'$. Then 
$z$ belongs to~$Y$, and by the Mean Value Theorem, satisfies
\be
\dot z -\Delta z + 3 \gamma y^2 z = {\sum_{i=0}^m u_i b_i } z,
\ee
where $y(x,t)\in [y'(x,t),y''(x,t)]$, for a.a.
$(x,t)\in Q$, so $y^2 \in L^{\infty}(0,T;L^3(\Om))$.
Multiplying by $z$ and integrating in space we deduce that for a.a. $t\in (0,T)$
\be
\half \ddt \|z(\cdot,t)\|^2_2 \leq {\sum_{i=0}^m \|b_i\|_\infty 
|u_i(t)|} \|z(\cdot,t)\|^2_{L^2(\Omega)}.
\ee
By Gronwall's Lemma, $z=0$, 
which gives the desired uniqueness property.
\end{proof}

\begin{remark}
We can similarly obtain, for any $n\in \cN,$ the estimates
\eqref{state-equ-estimA.l-1} and \eqref{state-equ-estim.l-2}
but with $\|y_0\|^4_{4} $
instead of $\|y_0\|^4_{H^1_0(\Om)}$.
\end{remark}
}\fi
}\fi
\if{
\begin{corollary}
\label{yCH10}
For each $u\in \Uspace,$ the corresponding state $y[u]$ belongs to the space $C([0,T];H_0^1(\Omega))$.
\end{corollary}

\begin{proof}
The result follows from the continuous embedding $Y\subset C([0,T];H_0^1(\Omega))$ by Lions-Magenes \cite{ }XXX.
\end{proof}
}\fi

\if{
\begin{lemma}
\label{LemmaEstimatez}
For any 
$\fb \in L^1(0,T;L^2(\Omega))$ and $z_0 \in H^1_0(\Om)$,
 the equation
\be
\label{lineq}
\left\{
\begin{split}
& \dot z + A z  = \fb\quad \text{in } Q;\\
& z=0\,\, \text{on } \Sigma,\quad z(x,0) = 0\,\, \text{in } \Omega,
\end{split}
\right.
\ee 
has a unique solution that verifies
\be
\label{estimatez}
\|z\|_{L^\infty(0,T;L^2(\Omega))} \leq e^{{\frac{T}{2}+\sum_{i=0}^m\|\ub_i\|_1 \|b_i\|_\infty}} \|\bar f\|_{L^1(0,T;L^2(\Omega))}.
\ee
\end{lemma}

\begin{proof}
Indeed, multiplying \eqref{lineq} by $z(x,t)$
and integrating over space we obtain that for a.a. $t\in (0,T)$
\be
\label{lin-se-est}
\begin{split}
\half \frac{\dd}{\dd t} \norm{z(\cdot,t)}{L^2(\Om)}^2 &+  \norm{\nabla z(\cdot,t)}{L^2(\Om)}^2 +
3 \gamma \| \yb(\cdot,t) z(\cdot,t)\|_{L^2(\Om)}^2
\\
&
=\int_\Om  z(x,t) \left( \fb(x,t) +{\sum_{i=0}^m \ub_i(t)\cdot b_i(x)} z(x,t)\right)\dd x.
\end{split}
\ee
The r.h.s. of \eqref{lin-se-est} can be bounded above by 
\be
\half \|\fb(\cdot,t)\|^2_{L^2(\Om)} 
+
\left( \half + {\sum_{i=0}^m\|\ub_i\|_2 \| b_i\|_\infty}\right) \|z(\cdot,t)\|^2_{L^2(\Om)}.
\ee 
Then we deduce the estimate \eqref{estimatez} with Gronwall's Lemma.
\end{proof}
}\fi

\begin{theorem}
\label{staequ-wp-th}
The mapping
$u\mapsto y[u]$ is of class $C^\infty$, from $L^2(0,T)^m$ to $Y$.
\end{theorem} 

\if{\begin{proof}We refer to \cite{ABK-PartI}.
\if{
We already proved in Lemma \ref{state-equ-estim.l}
that the state equation has
 a unique solution $y[u]$ in $Y$.
Let us define the mapping 
$\calf: \Uspace \times Y \rar L^2(Q)\times H^1_0(\Om)$, 
\be
\calf(u,y):=\left(
\dot y - \Delta y +\gamma y^3 -f  - y\sum_{i=0}^m
u_i b_i \,, \,y(\cdot,0)-y_0 \right).
\ee
The state equation can be rewriten as $\calf(u,y)=0$.
Since $y\in C([0,T]; H^1_0(\Om))$ in view of Lemma \ref{state-equ-estimA.l} and $H^1_0(\Om)$ is continuously embedded in $L^6(\Om)$ (given that $n\leq 3$), by Lemma \ref{lem-uby}, $\calf$ is of class $C^\infty$. 
By the Implicit Function Theorem, the conclusion holds,
provided that we check that the linearized
equation \eqref{lineq} with arbitrary r.h.s. $\bar f$ is well-posed. 
We know from Lemma \ref{LemmaEstimatez} above that, given any 
$\fb \in L^2(Q)$ and $z_0 \in H^1_0(\Om)$,
 \eqref{lineq} has a unique solution $z\in L^{\infty}(0,T;L^2(\Om))$. 
 \if{For $\yb^2 - \sum_{i=1}^m \ub_ib_i\in L^{\infty}(0,T;L^3(\Om))$ we can apply \cite[lemma 2.1]{MR2683898} XXX\footnote{However, here $\ub$ is only $L^2(0,T)^m$! AK}
So with \eqref{lineq} we deduce a bound of \footnote{{XXX: check this} Relation of $\fh$ not
really clear. Maybe we add: Integrating in time \eqref{lin-se-est} we obtain an a priori estimate in $W(0,T)$. To apply Lemma 2.1 the $q$ defined there has to be in $L^{\infty}(0,T;L^3(\Om)$}.
\be
\fh := \dot z - \Delta z + 3 \gamma \yb^2 z 
\ee
in $L^2(Q)$. 
The conclusion follows then from 
\cite[lemma 2.1] {MR2683898}. 
}\fi}\fi
\end{proof}
}\fi


\if{We need the following results (see~
\cite[page 14]{MR712486} and 
\cite[Theorem 8.20.5]{MR0221256}, respectively): 
\be
\label{comp-inj-10}
\left\{ \ba{lll}
\text{For any $p\in [1,10)$, the following injection is compact:}
\\
Y \hookrightarrow L^p(0,T; L^{10}(\Om)), \;\; \text{when $n\leq 3$.}
\ea\right.
\ee 
\be
\label{dual-lp}
\left\{\ba{lll}
\text{If $Y$ is a reflexive Banach space and $p\in (1,\infty)$,}
\\
\text{the dual of $L^p(0,T;Y)$ is $L^q(0,T;Y^*)$, where $1/p+1/q=1$}.
\ea\right.
\ee
}\fi
\if{
\begin{lemma}
\label{lem-weak-cv}
The mapping $u\mapsto y[u]$ is 
sequentially weakly  continuous from the space 
$\Uspace$ into $Y$.
\end{lemma}

\begin{proof}
Taking $u_{\ell} \rightharpoonup \ub$ in $\Uspace$, we shall prove that $y_{\ell}\rightharpoonup \yb$ in $Y$, where $y_{\ell}:=y[u_{\ell}]$, $\yb:=y[\ub].$ We know that it is enough to check that any subsequence of  $y_{\ell}$ weakly converges to $\yb$  in $Y$. To do this, we prove that we can pass to the limit in each term of the state equation.

(a) By Lemma \ref{state-equ-estim.l}, $(y_{\ell})$ is bounded in $Y$.
Extracting if necessary a subsequence, we may assume that it 
weakly converges to some $\yh$ in $Y$. 
By \eqref{comp-inj-10}, $y_{\ell} \rar \yh$ in $L^6(Q)$ and, therefore, maybe for a subsequence, it converges almost everywhere in $Q$.
By \cite[Lemma~1.3, Chapter 1]{MR0259693}, we have that
$y^3_{\ell} \rightharpoonup \yh^3$ in $L^2(Q)$. 

 (b) We will next see how to choose $\nu\in (1,2)$,
such that
\be
\label{uk-yk-nu}
\text{$(u_{\ell}\cdot b) y_{\ell}  \rightharpoonup (\ub\cdot b) \yh$ 
in $L^\nu(0,T; L^{2}(\Om))$.}
\ee
By \eqref{dual-lp}, the dual of 
$L^\nu(0,T; L^{2}(\Om))$
is
$L^{\nu'}(0,T; L^{2}(\Om))$,
where
$1/\nu+1/\nu'=1$. 
So, \eqref{uk-yk-nu}
is equivalent to 
\be
\int_Q (u_{\ell}\cdot b) y_{\ell} \varphi 
\rar 
\int_Q (\ub\cdot b) \yh \,\varphi,
\quad
\text{for every $\varphi \in L^{\nu'}(0,T; L^{2}(\Om))$.}
\ee
It suffices to check that, for $i=0,\dots,m$,
\be
\label{w-cv-ubyh}
\int_0^T  u_{\ell i} \int_\Om b_i y_{\ell} \varphi 
\rar 
\int_0^T \ub_i \int_\Om b_i  \yh \varphi.
\ee
Since $u_{\ell} \rightharpoonup \ub$ in $L^2(0,T)$,
it suffices that
$\int_\Om b_i y_{\ell} \varphi \rar 
\int_\Om b_i  \yh \varphi$
in $L^2(0,T)$, that is,
\be
\int_0^T \left( \int_\Om b_i (y_{\ell}-\yh)\varphi \right)^2 \rar 0.
\ee
We have that
\begin{multline*}
\int_0^T 
\left( \int_\Om b_i (y_{\ell}-\yh)\varphi \right)^2 
\\
\leq
\| b_i \|^2_\infty 
\int_0^T \| y_{\ell}(\cdot,t)-\yh(\cdot,t)\|^2_{L^2(\Omega)}\| \varphi(\cdot,t)\|^2_{L^2(\Omega)}
\leq
\| b_i \|^2_\infty A_y A_\varphi,
\end{multline*}
where, for $s=\nu'/2$, {and $s'$ satisfying} $1/s+1/s'=1$, by H\"older's inequality, 
\be
\ba{lll}
A_\varphi & \leq c \disp
\left( \int_0^T \| \varphi\|^{2s}_2 \right)^{1/s}
 = c 
\left( \int_0^T \| \varphi\|^{\nu'}_2 \right)^{2/\nu'}
 \leq c 
\| \varphi\|^2_{L^{\nu'}(0,T;L^2(\Om))},
\vspace{1mm} \\ 
A_y &  \leq c  \disp
\left( \int_0^T \| y_{\ell}-\yh\|^{2s'}_2 \right)^{1/s'}
 \leq c 
\| y_{\ell}-\yh\|^2_{L^{2s'}(0,T;L^2(\Om))},
\ea\ee
for some constant $c>0$.
By \eqref{comp-inj-10}, 
$A_y\rar 0$, provided that $2s' <10$.
We could take for instance $s'=4$, then $s=4/3$, 
$\nu'=8/3$, $\nu=8/5$. 
Then we have proved that \eqref{uk-yk-nu} holds.

\if{\\ (c) 
{XXX Why  do we need mild solution here and not just
pass to the limit inthe state equation ?}
Remembering that the operator $\cala :=-\Delta$ associated with the
homogeneous Dirichlet condition generates a strongly continuous
semigroup in $L^2(\Om)$ and  $\fh \in L^1(0,T;L^2(\Om))$. 
The {\em mild solution} of 
\be
\dot y + \cala y = \fh; \quad y(0)=y_0,
\ee
denoted by $y[y_0,\fh]$,
is such that the linear mapping 
$(y_0,\fh)\mapsto y=y[y_0,\fh]$ is
well-defined and continuous
$L^2(\Om) \times L^1(0,T;L^2(\Om)) \rar C([0,T];L^2(\Om))$,
with the representation 
\be
y[\fh] = e^{- t \cala} y_0 + \int_0^t e^{- (t-s) \cala} \fh(s) \dd s.
\ee}\fi

By steps (a)-(b), we can pass to the limit in the 
weak formulation, and obtain (due to the uniqueness of solution) that $\yh=\yb.$ 
The conclusion follows.
\end{proof}
}\fi

\begin{theorem}
{\rm (i)}
The function $u\mapsto J(u,y[u]),$ from $L^2(0,T)^m$ to $\RR$,
is weakly sequentially l.s.c.
{\rm (ii)}
The set of solutions of the optimal control problem 
\eqref{P} 
is weakly sequentially closed in $L^2(0,T)^m.$
{\rm (iii)}
If \eqref{P} has a bounded minimizing sequence, 
the set of solutions of \eqref{P} 
is non empty. This is the case in particular if 
\eqref{P} is admissible and $\Uad$ is a nonempty, closed bounded convex subset of $L^2(0,T)^m$.
\end{theorem}
\if{\begin{proof}We refer to \cite{ABK-PartI}.
\if{(i) 
Combine Lemma \ref{lem-weak-cv}
and the fact that the cost function $J$
is  continuous and convex on
$\Uspace \times Y$, hence
it is also weakly lower semicontinuous over this product space. 
\\ (ii)
Let
 $(u_{\ell}) \subset L^2(0,T)^m$
be a sequence of solutions weakly converging to $\ub \in L^2(0,T)^m,$
with associated states $y_{\ell}$.
By  Lemma \ref{lem-weak-cv}, 
$(y_{\ell})$ weakly converge in $Y$ to the state $\yb$
associated with $\ub$ and, by point (i),
$J(\ub,\yb) \leq \liminf_{\ell} J(u_{\ell},y_{\ell})$. 
This lower limit being nothing but the value of
problem \eqref{P},
the conclusion follows. 
\\  (iii)
By the previous arguments,
a weak limit of a minimizing sequence
is a solution of~\eqref{P}.
This weak limit exists iff the sequence is bounded.
This concludes the proof. }\fi
\end{proof}
}\fi




\section{First order analysis}\label{sec:2}
Here, we recall some properties from \cite{ABK-PartI}.

Throughout the section, $(\ub,\yb)$ is a trajectory of problem \eqref{P}.
 We recall the hypotheses \eqref{HypSpaces1} and \eqref{HypSpaces2} 
on the data. 

We fix a trajectory 
$(\ub,\yb=y[\ub]).$ 
Let $A$ be linear continuous 
from $L^2(0,T; H^2(\Om))$ to $L^2(Q)$ such that, for each $z\in L^2(0,T; H^2(\Om))$ and $(x,t) \in Q,$
\be
\label{lin-state-equ}
(Az)(x, t) :=   -\Delta z(x,t) +3\gamma \yb(x,t)^2 z(x,t)- \sum_{i=0}^m \ub_i(t) b_i(x) z(x,t).
\ee

\subsection{The linearized state equation}
The {\em linearized state equation} at $(\ub,\yb)$ is given by
\be
\label{lineq2}
 \dot z + A z  =  \sum_{i=1}^m v_i b_i \yb\quad \text{in } Q;\quad 
z=0\,\, \text{on } \Sigma,\quad 
z(\cdot,0) = 0\,\, \text{on } \Omega.
\ee 
For $v\in \Uspace$, equation \eqref{lineq2} above possesses a unique solution 
$z[v]\in Y$ 
and the mapping $v \mapsto z[v]$ is
linear from $\Uspace$ to $Y.$
Particularly, the following estimate holds:
\be
\label{estlineq2}
\|z\|_{L^\infty(0,T;L^2(\Omega))}
\leq
{M_1}   \sum_{i=1}^m \|b_i\|_\infty \|v_i\|_1,
\ee
where $M_1 := e^{\frac{T}{2} + \sum_{{i=0}}^m  \|\ub_i\|_1\|b_i\|_\infty} \|\yb\|_{L^\infty(0,T;L^2(\Omega))}.$
\if{
\begin{proof}
Immediate consequence of 
Lemma \ref{LemmaEstimatez}.
\end{proof}
}\fi

\subsection{The costate equation}
\label{costate_equation}
\if{
For $\mu\in BV(0,T)^q$ its distributional derivative $\dd\mu$ is in the {\em space $\mathcal{M}(0,T)$ of finite Radon measures.}
And, conversely, any element $\dd \mu\in \mathcal{M}(0,T)$ can be
identified with a function $\mu$ of bounded
variation that vanishes at time $T.$
Let us consider the set of positive finite Radon measures $\mathcal{M}_+(0,T)$ and identify it with the set
\be
BV(0,T)_{0,+}^q := \{ \mu \in BV(0,T)^q; \;\mu(T) =0, \dd \mu \geq 0 \}. 
\ee
}\fi
The {\em generalized Lagrangian} of problem $(P)$ is,
choosing the multiplier of the state equation to be 
$(p,p_0) \in L^2(Q) \times H^{-1}(\Om)$ and taking $\beta\in\RR_+$, 
$\dd\mu\in \mathcal{M}_+(0,T),$
\be\label{Lagrangian}
\begin{split}
&\call [\beta,p,p_0, \mu](u,y) :=  \beta J(u,y)  -
\la p_0, y(\cdot,0) - y_0 \ra_{H^1_0(\Om)} 
\\
&+\int_Q p \Big(\Delta y(x,t) -\gamma y^3(x,t) +
f (x,t) + \sum_{i=0}^m u_i(t) b_i(x) y(x,t) - \dot y(x,t)\Big) {\rm d} x
{\rm d} t
\\
&+ \sum_{j=1}^q \int_0^T g_j(y(\cdot, t)) {\rm d} \mu_j(t).
\end{split}
\ee
The  {\em costate equation} is the condition of stationarity  of the 
Lagrangian $\call$ {with respect to the state} that is, for any $z\in Y$:
\begin{multline}
\label{costate-eq}
 \int_Q p(\dot z + Az ){\rm d}x {\rm d}t 
+ \la p_0, z(\cdot,0) \ra_{H^1_0(\Om)} =
 \sum_{j=1}^q \int_0^T \int_\Omega c_j z {\rm d} x {\rm d} \mu_j(t)  
 \\ +
\beta \int_Q  (\yb-y_d) z {\rm d}x {\rm d}t
 + \beta \int_\Omega (\yb(x,T)-y_{dT}(x)) z(x,T)  {\rm d}x.
 \end{multline}
To each $(\varphi,\psi)\in L^2(Q)\times H^1_0(\Om)$,
let us associate $z=z[\varphi,\psi] \in Y$, the unique solution of 
\be
\dot z +Az = \varphi; \quad z(\cdot,0)= \psi. 
\ee
{Since this mapping is onto},
the costate equation \eqref{costate-eq} can be rewritten, for 
$z=  z[\varphi,\psi]$ and arbitrary
$(\varphi,\psi)\in L^2(Q)\times H^1_0(\Om)$, as (see \cite[Equation (3.7)]{ABK-PartI})
\begin{multline}
\label{costat-eq}
  \int_Q p \varphi {\rm d}x {\rm d}t 
+ \la p_0, \psi \ra_{H^1_0(\Om)} = 
 \sum_{j=1}^q \int_0^T \int_\Omega c_jz {\rm d} x {\rm d} \mu_j(t) , \\
   +
\beta \int_Q  (\yb-y_d) z {\rm d}x {\rm d}t
 + \beta \int_\Omega (\yb(x,T)-y_{dT}(x)) z(x,T)  {\rm d}x.
\end{multline}
\if{The r.h.s. of \eqref{costat-eq} {can be seen as a  linear continuous form on the pairs $(\varphi,\psi)$ of the space}  $L^2(Q) \times H^1_0(\Om)$.
By the Riesz Representation Theorem, {there exists a unique $(p,p_0) \in L^2(Q) \times H^{-1}(\Om)$ satisfying \eqref{costat-eq}, that means, there is a unique solution of the costate equation}.}\fi

\if{
\begin{remark}
\label{p-smooth}
If $p$ is smooth enough, we can integrate 
by parts in time and we have the 
initial-terminal conditions 
\be
\label{p-smooth1}
p(\cdot,T) = \beta (\yb(\cdot,T)-y_{dT}(\cdot)), 
\quad 
p(0) = p_0. 
\ee
\end{remark}
}\fi
Next consider the {\em alternative costates}
\be
\label{p1}
{p^1 := p+ \sum_{j=1}^q c_j \mu_j};
\quad
p^1_0 := p_0 + \sum_{j=1}^q c_j\mu_j(0),
\ee
where $\mu \in BV(0,T)^q_{0,+}$ is the 
function of bounded variation
associated with~$\dd \mu$.
By
\cite[Lem. 3.2]{ABK-PartI},
$p^1 \in Y$ and $p^1(\cdot,0) = p^1_0$.
Therefore $p(\cdot,0)$ makes sense as an element of 
$H^1_0(\Om)$, and it follows that
$p(\cdot,0) = p^1(\cdot,0) - \sum_{j=1}^q c_j \mu_j(0) = p_0$.

\if{
\begin{lemma}
\label{costate-eq-reg.l}
Let $(p,p_0,\mu)\in L^2(Q)\times H^{-1}(\Omega)
\times BV(0,T)^q_{0,+}$ satisfy
\eqref{costat-eq}, and let
$(p^1,p^1_0)$
be given by \eqref{p1}.
Then $p^1\in Y$, it satisfies $p^1(0)=p^1_0$,
    and it is the unique solution of
\be
\label{costat-eq4}
- \dot p^1 + A p^1 = \beta  (\yb-y_d) 
 + \sum_{j=1}^q \mu_j A c_j,
\quad
p^1(\cdot,T) = \beta (\yb(\cdot,T)-y_{dT}).
\ee
Moreover,
$p(x,0)$ 
  and $p(x,T)$ are well-defined as elements of $H^1_0(\Omega)$ in view of \eqref{p1}, 
  and we have 
\be\label{equationp1}
p(\cdot, 0)=p_0,\quad p(\cdot, T)=\beta ( \yb(\cdot, T) - y_{dT} ).
\ee
\end{lemma}
\if{
\begin{proof}
Let $z\in Y$. Note that, for $1\leq j \leq q$, the function
$t\mapsto\int_\Om c_j(x)z(x,t) \dd x$,
belongs to $W^{1,1}(0,T)$ and is, therefore,
of bounded variation. 
Using the integration by parts formula for the product of 
scalar functions with
bounded variation, one of them being continuous
(see e.g. \cite[Lemma 3.6]{MR2683898}),  
and taking into account the fact that $\mu_j(T)=0$,
we get that, for $\psi=z(\cdot,0)$,
\be
\label{costat-eq2}
\sum_{j=1}^q\int_Q  c_j \mu_j \dot z \dd x \dd t
+ \sum_{j=1}^q \mu_j(0) \la c_j,\psi\ra_{L^2(\Omega)}
= -
 \sum_{j=1}^q\int_0^T \int_\Om c_j z {\rm d} x {\rm d} \mu_j(t).
\ee 
By the definition \eqref{p1} of the alternative costate, the latter equation can be rewritten as
\be
\label{costat-eq2r}
\int_Q  (p^1-p) \dot z \dd x \dd t
+\la p^1_0 - p_0,\psi\ra_{H^1_0(\Omega)}
= -
 \sum_{j=1}^q\int_0^T \int_\Om c_j z {\rm d} x {\rm d} \mu_j(t).
\ee 
Now adding 
\eqref{costat-eq} and \eqref{costat-eq2r}, as well as the identity
\be
\int_Q (p^1-p)  A z 
=
\int_Q \sum_{j=1}^q c_j \mu_j  A z
\ee
we obtain, since $\varphi=\dot z + Az$, 
that
(implicitly identifying, as usual, $L^2(\Om)$ with its dual)
\begin{multline}
\label{costat-eq3}
 \int_Q p^1 \varphi {\rm d}x {\rm d}t 
+ \la p^1_0, \psi \ra_{H^1_0(\Om)}
\\= \beta \int_Q  (\yb-y_d) z {\rm d}x {\rm d}t  
 + \beta\int_\Omega (\yb(x,T)-y_{dT}(x)) z(x,T)  {\rm d}x   + \int_Q \sum_{j=1}^q c_j \mu_j  A z.
\end{multline} 
Since $A$ is symmetric, 
using \eqref{HypSpaces2},
we see that $p^1$ is solution in $Y$ of \eqref{costat-eq4};
 the solution of the latter being clearly unique.
Multiplying \eqref{costat-eq4} by $z\in Y$ and integrating over~$Q$,
with an integration by parts of the term with $\dot p^1z$,
we recover (using \eqref{p1}) equation \eqref{costat-eq3}
implying that
$p^1(x,0) = p^1_0(x)$ for a.a. $x$ in $\Om$.
It is easy to prove that,
conversely, any solution of \eqref{costat-eq3}
is solution of \eqref{costat-eq4}.

\if{Finally, $p^1$ is in $W(0,T)\subset C(0,T;H)$ and $\mu$ has bounded
variation, so we get by \eqref{p1} the condition for $p$ at initial and final point.}\fi
Since $p^1$ and 
$c_j \mu_j$ belong to $L^{\infty}(0,T;H^1_0(\Om))$, by \eqref{p1} also $p$ has this regularity.
Use \eqref{p1} again,
the final condition on $p^1$ and the fact that
$\mu(T)=0$ to get
the second relation of
\eqref{equationp1}. Furthermore, we have
\be
p_0=p^1(\cdot, 0)-  \sum_{j=1}^q c_j\mu_j(0)=p(\cdot,0).
\ee
\end{proof} 
}\fi
}\fi

\begin{corollary}\label{cor:reg-p}
\if{
The solution $p$ of \eqref{costate-eq} belongs to 
$BV(0,T;L^2(\Om))$
and 
\be
\label{equationp1}
p(T)=\beta\big( \yb(T) - y_{dT}\big).
\ee
}\fi
If $\mu \in H^1(0,T)^q$ then 
$p\in Y$ and
\be
\label{equationp2}
-\dot p +Ap = \beta(\yb-y_d) + \sum_{j=1}^q c_j \dot{\mu}_j.
\ee
\end{corollary} 

\if{
\begin{proof}
This follows immediately from \eqref{p1} and \eqref{costat-eq4}. \end{proof}
}\fi

\section{An example} 
We recall an example from \cite[Appendix B]{ABK-PartI} , and show that it satisfies the
sufficient condition for quadratic growth
(condition (a) of Theorem~\ref{ThmSC}).

We consider the following setting: Let $\Om = (0,1),$  and denote by 
$c_1(x):=\sqrt 2 \sin \pi x$ the first (normalized) eigenvector
of the Laplace operator. We assume that $\gamma=0$, 
 the control is scalar ($m=1$), 
$b_0 \equiv 0$ and $b_1 \equiv 1$ in $\Omega,$
and that $f \equiv 0$ in $Q.$ 
Then the state equation with initial condition $c_1$ reads
\be
\label{eqex}
\dot y(x,t) - \Delta y(x,t) = u(t) y(x,t); \;\quad (x,t) \in (0,1)\times (0,T),
\quad y(x,0) = c_1(x),\quad x\in \Om.
\ee 
The state satisfies
$y(x,t) = y_1(t) c_1(x)$, where $y_1$ is solution of 
\be
\dot y_1(t) + \pi^2 y_1(t) = u(t) y_1(t); \;\quad t \in (0,T),
\quad y_1(0) = y_{10}=1. 
\ee
We set $T=3$ and consider the state constraint \eqref{stateconstraint} with $q=1$
and $d_1:=-2,$
and the cost function \eqref{cost} with $\alpha_1=0$. 
%
%
The state constraint reduces to
\be
y_1(t) \leq 2,\quad t\in [0,3].
\ee
As target functions we take
$y_{dT}:= c_1$  and
$y_d(x,t) := \yh_d(t) c_1(x)$ with
\be
\yh_d(t) :=
\left\{ \ba{lll}
1.5e^t &\quad \text{for }t\in (0, \log 2),
\vspace{1mm} \\ 
3 &\quad \text{for }t\in (\log 2, 1),
\vspace{1mm} \\ 
4- t &\quad \text{for }t\in (1, 3).
\ea\right.\ee
 We assume that the lower and upper bounds for the control are  $\umin := -1$ and $\umax :=\pi^2+1$.
The optimal control is given by
\be\label{ex-control}
\ub(t) :=
\left\{ \ba{ll}
\umax & \quad \text{for } t\in (0, \log 2), \vspace{1mm}  \\ 
\pi^2 & \quad \text{for } t\in (\log 2, 2),  
\vspace{1mm} \\ 
 \pi^2 -1/\hat{y}_d & \quad \text{for } t\in (2, 3).
\vspace{1mm} \\ 
\ea\right.\ee
and the optimal state by
\be
\yb_1(t) :=
\left\{ \ba{lll}
e^t & \quad\text{for } t\in (0, \log 2), &
\vspace{1mm} \\ 
 2 & \quad\text{for } t\in (\log 2,2), &  
\vspace{1mm} \\ 
 4-t & \quad \text{for } t\in (2, 3). &
\vspace{1mm} \\ 
\ea\right.\ee
The above control is feasible. The trajectory $(\ub,\yb)$ is optimal.
The costate equation is 
\be
-\dot p +Ap = c_1 (\yb_1-\hat y_d) + c_1\dot{\mu}_1,\quad p(\cdot,T)= \yb(T)-y_{dT} = 0.
\ee
Since $\yb$ and $y_d$ are colinear to $c_1$,
it follows that $p(x,t) = p_1(t) c_1(x)$, and 
\be
-\dot p_1 + \pi^2 p_1 = \ub p_1 + \yb_1 - \yh_{d}  + \dot{\mu}_1; \quad p_1(3)=0.
\ee
Over $(2,3)$, $\dot{\mu}_1=0$ (state constraint not active) and 
$\yb_1 = \yh_d$, therefore $p_1$ and $p$ identically vanish.
Over $(\log 2,2)$, $\ub$ is out of bounds and therefore
\be
0 = \int_\Om p(x,t) \yb(x,t) = p_1(t) \yb_1(t) \int_\Om c_1(x)^2 = 2p_1(t).
\ee
It follows that $p_1$ and $p$ also vanish on $(\log 2,2)$ and that
\be
\dot \mu_1 = -( \yb_1 - \yh_{d} ) >0, \quad \text{a.a. $t\in (\log 2,2)$. }
\ee
Over $(0,\log 2),$ the control attains its upper bound, then
\be
-\dot p_1 = p_1 - \half e^t
\ee
with final condition $p_1(\log 2)=0$, so that
\be
p_1(t) = \frac{e^t}{4} - e^{-t}.
\ee
As expected, $p_1$ is negative.

\begin{lemma}
The hypothesis (a) of Theorem \ref{ThmSC} holds.
\end{lemma}
\begin{proof}
(i)
This has been obtained in part I \cite{ABK-PartI}.
Note that the multiplier is unique.
\\ (ii) 
We check the Legendre form condition.
For this, we apply the Goh transformation to the example. For $(v,z)$ solution of the linearized state equation we define
\be
B:=\yb b = \yb_1(t) c_1(x); \quad \xi:=z-Bw = (z_1 -\yb_1 w) c_1
\ee
and we observe that $\xi=\xi_1 c_1$ is solution of
\be
\dot \xi + A \xi = - (AB+\dot B) w; \quad \xi(0)=0;
\ee
where
\be
AB+\dot B = (\pi^2 - \ub) B +\dot \yb_1 c_1
=
( (\pi^2 -\ub) \yb_1  +\dot \yb_1 ) c_1
\ee
so that
\be
\label{dyn_xi}
\dot \xi_1 + (\pi^2-\ub)\xi_1 = B^1 w, \; B^1 := (\pi^2 -\ub) \yb_1  +\dot \yb_1.
\ee
For checking the Legendre condition 
((ii) of Theorem \ref{ThmSC}), 
we have to check the uniform positivity
of the coefficient 
of $w^2$ in $\hat Q$.
This trivially holds on the second and third arcs, since then $p$
and therefore $\chi$ vanish, so that the coefficient of $w^2$ reduces to 
$\int_\Om \kappa \yb^2 = \yb_1(t)^2 \geq 1$.
We now detail the computation for the first arc.
Replacing $z$ by $\xi + Bw$ in the quadratic form $\calq[p](v,z)$ we have
\be
\begin{aligned}
\tilde{Q}&= \int_{Q} ( (\xi+Bw)^2 + p v (\xi+Bw))\dd x \dd t + \int_\Om (\xi(\cdot,T)+B(\cdot,T)w(T))^2 \dd x.
\end{aligned}
\ee
For the second term in the first integral we have
\be
\begin{aligned}
\int_{Q}  p v (\xi+Bw)\dd x \dd t
&=
\int_0^T \left(p_1 \xi \ddt w + \half p_1 \yb_1 \ddt ( w^2) \right) \dd t 
\\ & =
-\int_0^T \left(\ddt (p_1\xi) w + \half \ddt (p_1 \yb_1) w^2 \right) \dd t + [\text{boundary-terms}].
\\ & =
-\int_0^T \left( p_1B^1 + \half \ddt (p_1 \yb_1) \right) w^2  \dd t + [\text{boundary-terms}].
\end{aligned}
\ee
Finally we obtain that over the first arc, the coefficient of $w^2$
in the integral term of $\tilde{Q}$ is $2+e^{2t}/4$.
It follows that $\widehat{\calq}[p](w,\xi[w])$ is a Legendre form.
\\(iii)
We check the uniform positivity condition.
Any $(w,h) \in PC_2^*$ is such that $w$ vanishes on the two first arcs,
and since the costate vanishes on the third arc we have that, using
$\yb(x,t) = (4-t)c_1(x)$
on the third arc
\be
\begin{split}
\widehat{\calq}[p,\mu](\xi,w,h)
&=
\int_2^3 \int_\Om (\xi(x,t)+(4-t) c_1(x) w(t))^2 \dd x \dd t +  \int_\Om (\xi(x,T)+ hc_1(x))^2 \dd x\\
&=
\int_2^3  (\xi_1(t)+(4-t)  w(t))^2\dd t +   (\xi_1(T)+ h)^2.
\end{split}
\ee
This is a Legendre form over $L^2(2,3)$, and so it is coercive iff it has positive
values except at $0$. If the value is zero then
$w(t) = \xi_1(t)/(t-4)$ so that $\xi$ vanishes identically and therefore
$w$ also, and $h=0$.
The conclusion follows.
\end{proof}


\bibliographystyle{amsplain}
\bibliography{hjbx}

\end{document}